[1994/12/01]% LaTeX date must December 1994 or later

\documentclass[12pt,fleqn]{article}
\usepackage{amsfonts,amssymb,amsmath,amsthm}

\def\a{\mathfrak{a}}
\def\b{\mathfrak{b}}
\def\c{\mathfrak{c}}
\def\e{\mathfrak{e}}
\def\f{\mathfrak{f}}
\def\g{\mathfrak{g}}
\def\k{\mathfrak{k}}
\def\p{\mathfrak{p}}
\def\sl{\mathfrak{sl}}
\def\so{\mathfrak{so}}
\def\sp{\mathfrak{sp}}
\def\su{\mathfrak{su}}

\def\R{{\bf R}}
\def\N{{\bf N}}
\def\C{{\bf C}}
\def\F{{\bf F}}
\def\H{{\bf H}}
\def\SO{{\bf SO}}
\def\O{{\bf O}}
\def\SU{{\bf SU}}
\def\Sp{{\bf Sp}}
\def\SL{{\bf SL}}
\def\U{{\bf U}}

\def\Z{\hbox{$\mathcal{Z}$}}
\def\S{\mathcal{S}}
\def\diag{\mathop{\hbox{diag}}}
\def\Ad{\mathop{\hbox{Ad}}}
\def\ad{\mathop{\hbox{ad}}}
\def\to{\rightarrow}
\def\indic{\mathbbm{1}}

\def\DR{\Delta^{\R^d}}

\def\M{\mathcal{M}}

\newcommand{\eval}[2][\right]{\relax
  \ifx#1\right\relax \left.\fi#2#1\rvert}

\begin{document}
%% issueinfo

\newtheorem{theorem}{Theorem}[section]
\newtheorem{lemma}[theorem]{Lemma}
\newtheorem{definition}[theorem]{Definition}
\newtheorem{prop}[theorem]{Proposition}	
\newtheorem{cor}[theorem]{Corollary}
\newtheorem{example}[theorem]{Example}
\newtheorem{remark}[theorem]{Remark}

\def\a{\mathfrak{a}}
\def\b{\mathfrak{b}}
\def\c{\mathfrak{c}}
\def\e{\mathfrak{e}}
\def\f{\mathfrak{f}}
\def\g{\mathfrak{g}}
\def\k{\mathfrak{k}}
\def\p{\mathfrak{p}}
\def\sl{\mathfrak{sl}}
\def\so{\mathfrak{so}}
\def\sp{\mathfrak{sp}}
\def\su{\mathfrak{su}}

\def\R{{\bf R}}
\def\N{{\bf N}}
\def\C{{\bf C}}
\def\F{{\bf F}}
\def\H{{\bf H}}
\def\SO{{\bf SO}}
\def\O{{\bf O}}
\def\SU{{\bf SU}}
\def\Sp{{\bf Sp}}
\def\SL{{\bf SL}}
\def\U{{\bf U}}

\def\Z{\hbox{$\mathcal{Z}$}}
\def\S{\mathcal{S}}
\def\diag{\mathop{\hbox{diag}}}
\def\Ad{\mathop{\hbox{Ad}}}
\def\ad{\mathop{\hbox{ad}}}
\def\to{\rightarrow}
\def\indic{\mathbbm{1}}

\def\DR{\Delta^{\R^d}}

\def\M{\mathcal{M}}

\begin{center}
INTEGRAL KERNELS ON COMPLEX SYMMETRIC SPACES AND FOR THE DYSON BROWNIAN MOTION\\
P.{} Graczyk\footnote{LAREMA, UFR Sciences, Universit\'e d'Angers, 2 bd Lavoisier, 49045 Angers cedex 01, France,piotr.graczyk@univ-angers.fr} and P.{} Sawyer\footnote{Department of Mathematics and Computer Science, Laurentian University, Sudbury, Canada P3E 2C6, psawyer@laurentian.ca}
\end{center}


\begin{abstract}
In this article, we consider flat and curved Riemannian symmetric spaces in the complex case and we study their basic integral kernels, in potential and spherical analysis: heat, Newton, Poisson kernels and spherical functions, \emph{i.e.} the kernel of the spherical Fourier transform. 
 
We introduce and exploit a simple new method of construction of these $W$-invariant kernels by alternating sum formulas.
We then use the alternating sum representation of these kernels to obtain their asymptotic behavior.
We apply our results to the Dyson Brownian Motion on $\R^d$.
\end{abstract}

\section*{Thanks}
The first author thanks Laurentian University of Sudbury  for its hospitality  and financial support during his visits to Sudbury.  The second author thanks LAREMA for its hospitality and the R\'egion Pays de la Loire for its financial support on several occasions via the projects Matpyl, G\'eanpyl and D\'efimaths.

%\centerline{\bf Updated \filemodprintdate{\jobname}}

\section{Introduction and notations}

Analysis on Riemannian symmetric spaces of Euclidean type, also called flat symmetric spaces,
continues to develop in recent years (\cite{PGPS1, Helgason1, Wolf, Xu}). Its importance is due to its relationship with Dunkl analysis (\cite{dJ,DuX, Roesler}) together with the correspondence of the complex case with the parameter $k=1$,
in which  symmetric spaces of Euclidean type constitute the ``geometric case'', frequently used as a model case in most challenging open problems of Dunkl theory. The analysis on flat complex symmetric spaces coincides with Weyl group invariant
Dunkl analysis associated with multiplicity $k = 1$, see \cite{dJ}. In particular, the heat
kernel ${p_t^W(X,Y )}$ is a special case of the heat kernel in the Weyl group invariant Dunkl
setting. We  employ this intimate connection to Dunkl theory in our paper in Section \ref{asymp}
as one of main tools of the proof of main theorems. This connection appears also in Proposition \ref{25}.

Another important aspect of this paper is to apply
analysis  on symmetric spaces of Euclidean type to potential theory and
to stochastic analysis of Dyson Brownian Motion, one of the most important models of non-colliding particles, see the recent survey \cite{Katori}. We expect further applications of our results and techniques
to other  non-intersecting stochastic path problems related to
root systems and to  multivariate stochastic processes related to 
Laplace-Beltrami operators on symmetric spaces, to Dunkl Laplacians
and to Schr\"odinger operators, see the discussion in Section \ref{DYSON}. We thank an anonymous referee for pointing out to us such 
further stochastic applications.

The objective of this paper is to study  basic integral kernels,
in potential theory and spherical analysis: heat, Newton, Poisson kernels, Green function and spherical functions (\emph{i.e.} the kernel of the spherical Fourier 
transform), in the set-up of flat and curved symmetric spaces of complex type.

Our main results on the exact form and asymptotics of the heat, Poisson and Newton kernels (Theorems \ref{kernels}, \ref{Poiss}, \ref{Newt} and Corollaries \ref{Dyson} and \ref{Poiss_Dyson}) are crucial  for the future development of the potential theory on flat and curved symmetric spaces of complex type, and for the potential theory of the Dyson Brownian Motion. These results are a starting point of research and a source of conjectures for the corresponding kernels in the Weyl-invariant Dunkl setting (for the rank one case, refer to \cite{Graczyk}).

The main result on asymptotics of the spherical functions contained in Theorem \ref{Final} is important 
from the point of view of spherical analysis on symmetric spaces, because it  generalizes significantly the results of Helgason in \cite{Helgason1}, of Narayanan, Pasquale and Pusti in \cite{Narayanan} and of Schapira in \cite{Schapira}, for the flat and curved symmetric spaces in the complex case, cf.{} Remark \ref{pub}. 

We recall now some basic terminology and facts about symmetric spaces associated to Cartan motion groups.

Let $G$ be a semisimple Lie group and let $\g=\k\oplus\p$ be the Cartan decomposition of $G$. 
We recall the definition of the Cartan motion group and the flat symmetric space associated with the semisimple Lie group $G$ with maximal 
compact subgroup $K$. 
The Cartan motion group is the semi-direct product  
5

$G_0=K\rtimes \p$ where the multiplication is defined by $(k_1,X_1)\cdot(k_2,X_2)=(k_1\,k_2,\Ad(k_1)(X_2)+X_1)$. The associated flat symmetric space is then $M=\p\simeq G_0/K$ (the action of $G_0$ on $\p$ is given by $(k,X)\cdot Y=\Ad(k)(Y)+X$).

We tacitly identify $K$-invariant measures, functions, differential operators on $M$ with $W$-invariant measures etc.{} on $\a$.

The spherical functions for the symmetric space $M$ are then given by
\begin{equation}\label{psi}
\psi_\lambda(X)=\int_K\,e^{\lambda(\Ad(k)(X))}\,dk
\end{equation}
where $\lambda$ is a complex linear functional on $\a\subset \p$, a Cartan subalgebra of the Lie algebra of $G$. To extend $\lambda$ to $X\in\Ad(K)\a=\p$, one uses $\lambda(X)=\lambda(\pi_\a(X))$ where $\pi_\a$ is the orthogonal projection with respect to the Killing form (denoted throughout this paper by 
$\langle \cdot,\cdot\rangle$).
Note also that the spherical function for the symmetric space $G/K$ is given by
\begin{equation}\label{phi}
\phi_\lambda(g)=\int_K\,e^{(\lambda-\rho)(H(g\,k))}\,dk
\end{equation} 
where $\lambda$ is a complex linear functional on $\a$ and the map $H$ is defined via the Iwasawa decomposition 
of $G$, namely $g=k\,e^{H(g)}\,n\in K\,A\,N$ and $\rho=(1/2)\,\sum_{\alpha>0}\,m_\alpha\,\alpha$. 
Note that in \cite{Helgason1, Helgason2, Helgason3}, $\lambda$ is replaced by $i\,\lambda$.

Throughout this paper, we suppose that $G$ is a semisimple complex Lie group. The complex root systems are respectively 
$A_{n-1}$ for $n\geq 2$ (where $\p$ consists 
of the $n\times n$ hermitian matrices with trace 0), $B_n$ for $n\geq 2$ (where $\p=i\,\so(2\,n+1)$), $C_n$ for $n\geq 3$ (where $\p=i\,\sp(n)$) 
and $D_n$ for $n\geq 4$ (where $\p=i\,\so(2\,n)$) for the classical cases and the exceptional root systems $E_6$, $E_7$, $E_8$, $F_4$ and $G_2$.

Let $\Delta$ be the Laplace-Beltrami operator on $M$ and $\Delta^W$ its restriction to
$W$-invariant functions on $\a$ {where $W$ is the corresponding Weyl group}. Recall the formula 
\begin{equation}\label{operator}
\Delta^W f= \pi^{-1}\, \Delta^{\R^d} (\pi\, f),
\end{equation}
where $\pi(X)=\prod_{\alpha>0} \,\alpha(X)$ (see \cite[Chap.{} II, Theorem 5.37]{Helgason3}) in the Euclidean case.

In Section \ref{IntroKernel}, we introduce and exploit a simple new method of construction of important $K$-invariant kernels on the space $M$.

We show in Theorem \ref{kernels} that if ${\mathcal K}(X,Y)$ is
 an Euclidean kernel  (heat kernel, potential kernel, Poisson kernel,  \dots) for the Laplacian $\Delta^{\R^d}$, then  the
corresponding kernel  acting on $W$-invariant functions on $M$ is given by the alternating sum
\begin{equation}\label{alter}
{\mathcal K}^W(X,Y)=\frac{1}{|W|\,\pi(X)\,\pi(Y)} \,\sum_{w\in W}\,{\epsilon(w)} {\mathcal K}(X,w\cdot Y).
\end{equation}

Here and in Theorem \ref{kernels} below, ${\mathcal K}(X,Y)$ is an Euclidean kernel on the Cartan subalgebra $\a$ which is isomorphic to $\R^d$ where $d$ is the dimension of $\a$ and with the underlying scalar product being the Killing form on $\a$.

The proof of Theorem \ref{kernels} is short and easy and    uses the simple form of the operator $\Delta^W$ given in \eqref{operator}.

It is  well-known that the spherical functions of the space can be written explicitly as such alternating sums (\cite[Chap.{} IV, Proposition 4.10]{Helgason3}).

The alternating sum formulas \eqref{alter}
also include determinantal formulas for transition probabilities $p_t^W(x,y)$ (equivalently, for heat kernels) of Karlin-McGregor type, proven  for Dyson Brownian Motions in Weyl chambers \cite{Grabiner}
and exploited in stochastic analysis (refer to \cite{KT,Koenig}).

The fact that alternating sums formulas \eqref{alter} are  true for many further analytic and stochastic kernels beyond spherical functions and heat kernels, was surprisingly not  published  or exploited (we asked experts of the field for an existing reference).  

The approach with formulas \eqref{alter}  will allow us to provide asymptotics   for  kernels ${\mathcal K}^W$, using
 our knowledge  of the kernels  ${\mathcal K}(X,Y)$  on $\R^d$ as given in Table \ref{Ker}.

In Section \ref{asymp}, we discuss the asymptotic behaviour of the Poisson kernel especially when one or both arguments are singular. These results translate well to the Newton kernel.

In Section \ref{spherical}, we compute asymptotics for the spherical functions $\psi_\lambda(Y)$ which can prove challenging when either $\lambda$ or $Y$ are singular (\emph{i.e.} such that at least one of the nonzero root vanish on $X$ or $Y$). Our results depend on a property we call ``Killing-max'' namely the property that for $X$, $Y\in\overline{\a^+}$, 
$\langle X,w\cdot Y\rangle=\langle X, Y\rangle$ if and only if $w\in W_X\,W_Y$ where $W_X=\{w\in W\colon w\cdot X=X\}$. It is known that this property is verified when either $X$ or $Y$ is non singular \cite{Helgason1}. We prove in Appendix \ref{KM}, using the classification of Lie algebras, that the Killing-max holds in almost all cases (only in the cases related to the root systems $E_6$, $E_7$ and $E_8$ is the question left unanswered).

We conclude with Section \ref{DYSON} where we
apply the previous results to the heat kernel and Poisson and Newton kernels for the Dyson Brownian Motion. \\

{\bf Acknowledgements.}
We thank M{}. Denkowski for  advice with Lemma \ref{LLL} and  J.-J.{} Loeb for useful discussions.
 We are grateful to both anonymous referees for their insightful comments and remarks that greatly helped to improve the paper.

%%%%%%%%%%%%%%%%%%%%%%
\section{Kernels on flat symmetric spaces in the complex case}\label{IntroKernel}

\subsection{Definitions}

We first recall the classical integral kernels on $\R^d$ in Table \ref{Ker}.

\begin{table}[h]
\begin{center}
{\small
\begin{tabular}{|l|l|l|}\hline
PDE&Kernel&Solution\\ \hline
$\displaystyle \left\lbrace\begin{array}{l}\DR\,u(X,t)=\frac{\partial~}{\partial t}\,u(X,t)\\
			\lim_{t\to0^+}\,u(X,t)=f(X)\end{array}\right.$
			&$\displaystyle p_t(X,Y)=\frac{e^{-\frac{|X-Y|^2}{4\,t}}}{(4\,\pi\,t)^{d/2}}$
			&$\displaystyle u(X,t)=\int_{\R^d}\,f(Y)\,p_t(X,Y)\,dY$\\
			$X\in\R^d$, $t>0$&&\\\hline
			$\displaystyle \left\lbrace\begin{array}{ll}\DR\,u(X)=f(X)&\hbox{on $\R^d$,}\\\hbox{$|u(X)|\to0$ as $X\to\infty$}\end{array}\right.$\\$f\in C_c(\R^d)$
			&$\displaystyle N(X,Y)=\Phi(X-Y),$
			&$\displaystyle u(X)=\int_{\R^d}\,f(Y)\,N(X,Y)\,dY$\\ \hline
			$\displaystyle \left\lbrace\begin{array}{ll}\DR\,u(X)=0&\hbox{on $B(X_0,r)$}\\u(X)=f(X)&\hbox{on $\partial B(X_0,r)$}\end{array}\right.$
			&$\displaystyle P(X,Y)=\frac{r^2-|X-X_0|^2}{w_d\,r\,|X-Y|^d}$
			&$\displaystyle u(X)=\int_{\partial B(X_0,r)}\,f(Y)\,P(X,Y)\,dY$\\\hline
			$\displaystyle \left\lbrace\begin{array}{ll}\DR\,u(X)=f(X)&\hbox{on $B=B(0,1)$}\\ u(X)=0&\hbox{on $\partial B$}\end{array}\right.$
			&$\displaystyle G_B(X,Y)=\Phi(X-Y)$
			&$\displaystyle u(X)=\int_B\,f(Y)\,G_B(X,Y)\,dY$\\
			&$\displaystyle -\Phi(|X|\,(Y-X/|X|^2))$&\\\hline
		\end{tabular}
		}
	\end{center}
	where $w_d=2\,\pi^{d/2}/\Gamma(d/2)$ (the surface area of a sphere of radius 1 in $\R^d$) and
	$\displaystyle
	\Phi(X)=\left\lbrace\begin{array}{ll}
	\frac{1}{2\,\pi}\,\ln |X|&\hbox{if $d=2$}\\
	\frac{1}{(2-d)\,w_d}\,|X|^{2-d}&\hbox{if $d\geq 3$}
	\end{array}\right..$
	\caption{The heat kernel $p_t$, the Newton kernel $N$, the Poisson kernel $P$ and the Green kernel $G_B$ for $\R^d$ \label{Ker}}
\end{table}

%%%%%%%%%%%%%%%%%%%%%%%%%%%%%
The integral kernels on the flat symmetric space $M$ are considered with respect to the invariant measure $\mu(dY)=\pi^2(Y)\,dY$ on $M$.
 Their definition is analogous to the classical $\R^d$ and Riemannian manifold case, with the $W$-invariance imposed on the operator,
boundary problem and solutions. The Dunkl-Poisson, Newton and Green kernels and their $W$-invariant versions
 were introduced and studied in \cite{Gallardo2} and \cite{Graczyk}.

\begin{definition}
We define a kernel  ${\mathcal K}^W(X,Y) $ for the operator $\Delta^W$ and a boundary
problem ${\mathcal P}$  as the fundamental solution of this PDE problem, which is $W$-invariant in $X$-variable, for each $Y$.
Equivalently,  ${\mathcal K}^W(X,Y) $ is  an  integral reproducing kernel for the $W$-invariant solutions of the  problem ${\mathcal P}$
and this kernel is $W$-invariant in $X$.
\end{definition}

The uniqueness of ${\mathcal K}^W(X,Y)$ may be deduced, as in the classical case, from the uniqueness of the spherical Fourier transform.   Another approach for the existence of Poisson, Newton and Green kernels  is available from the point of view of stochastic diffusion processes \cite{Chung}. Note that $W$-invariant Dunkl processes are diffusions.

\subsection{The method of alternating sums for constructing kernels on $M$}\label{flatFormulas}

This method will be introduced and used in the proof of Theorem \ref{kernels} below.

%%%%%%%%%%%%%%%%%%%%%%%%%%%%%
%%%%%%%%%%%%%%%%%%%%%%%%%%%%%%%%%%%%%%%%%%%%%%%%%%%%%%%%%

\begin{theorem}\label{kernels}
Let $M$ be a symmetric space of Euclidean type with $G$ a complex simple Lie group of rank $d$. Then the following formulas hold 
for $X$, $Y\in\a$, a Cartan subalgebra associated with $M$.
\begin{enumerate}

\item The heat kernel on $M$ is given by
\begin{align}\label{HeatComplexX}
p_t^W(X,Y)
&=\frac{1}{|W|\,(4\,\pi\, t)^{d/2}\,\pi(X)\,\pi(Y)}
\,\sum_{w\in W}\,{\epsilon(w)}{e^{-{\frac{|X-w\cdot Y|^{2}}{4\,t}}}}
\end{align}

\item The Newton kernel on $M$ is given by
\begin{align}\label{NewtonComplexX}
N^W(X,Y)&=\frac{1}{2\,\pi\,|W|\,\pi(X)\,\pi(Y)}\,\sum_{w\in W}\,\epsilon(w)\,\ln|X-w\cdot Y|~~ \hbox{when $d=2$},\\
N^W(X,Y)&=\frac{1}{|W|\,(2-d)\,w_d\,\pi(X)\,\pi(Y)}
\,\sum_{w\in W}\,\frac{\epsilon(w)}{|X-w\cdot Y|^{d-2}}~~ \hbox{when $d\geq3$}.\nonumber
\end{align}
	
\item The Poisson kernel of the open unit ball $B$ is given for $X\in B$ and $Y\in\partial B$ by
\begin{align}\label{PoissonComplexX}
P^W(X,Y) &=\frac{1-|X|^2}{|W|\,w_d\,\pi(X)\,\pi(Y)} \,\sum_{w\in W}\,\frac{\epsilon(w)}{|X-w\cdot Y|^{d}}
\end{align}
	
\item The Green function of the unit ball is given by
\begin{align}\label{GreenComplexX}
G^W_B(X,Y) &=\frac{1}{|W|\,\pi(X)\,\pi(Y)} \,\sum_{w\in W}\,\epsilon(w) G_B(X,Y),
\end{align}
where $G_B(X,Y)$ is the classical Green function of the unit ball $B$ in $\R^d$ (refer to Table \ref{Ker}). 
\end{enumerate}
\end{theorem}
\begin{proof}
It is based on the following steps:
\begin{enumerate}
\item Write a kernel on $\R^d$ where $d$ is the rank of $M$; 
\item Exploit formula \eqref{operator};
\item Apply the $W$-invariance (the kernels on $M$ must be $W$-invariant).
\end{enumerate}

We give the proof in the Poisson kernel case; the other proofs are similar. The Poisson kernel of $B(0,1)$ in the Euclidean case is
\begin{align*}
P(X,Y) &=\frac{1-|X|^2}{w_d\,|X-Y|^{d}}.
\end{align*}
	
If $u$ is harmonic with respect to $\Delta^W$ then $\pi\,u$ is harmonic with respect to $\Delta^{\R^d}$. Hence 
\begin{align*}
\pi(X)\,u(X)&=\int_{\partial B}\,\frac{1-|X|^2}{w_d\,|X- Y|^{d}}\,\pi(Y)\,f(Y)\,dY.
\end{align*}
	
This is equivalent to
\begin{align*}
u(X)&=\int_{\partial B}\,\frac{1-|X|^2}{w_d\,\pi(X)\,\pi(Y)}\,\frac{1}{|X- Y|^{d}}\,f(Y)\,\pi(Y)^2\,dY.
\end{align*}
	
The reproducing kernel $\frac{1-|X|^2}{w_d\,\pi(X)\,\pi(Y)}\,\frac{1}{|X- Y|^{d}}$ is not $W$-invariant. We write the last equation 
$|W|$ times, replacing $X$ by $w\cdot X$
\begin{align*}
u(X)&=u(w\cdot X)=\int_{\partial B}\,\frac{1-|X|^2}{w_d\,\pi(w\cdot X)\,\pi(Y)}\,\frac{1}{|w\cdot X- Y|^{d}}\,f(Y)\,\pi(Y)^2\,dY\\
&=
\int_{\partial B}\,\frac{1-|X|^2}{w_d\,\pi(X)\,\pi(Y)}\,\frac{\epsilon(w)}{|X- w\cdot Y|^{d}}\,f(Y)\,\pi(Y)^2\,dY
\end{align*}
and we	 sum up the $|W|$ equations. We obtain
\begin{align*}
u(X)=\frac{1}{|W|\,w_d}\,\int_{\partial B}\,\frac{1-|X|^2}{\pi(X)\,\pi(Y)}\,\sum_{w\in W}\,\frac{\epsilon(w)}{|X- w\cdot Y|^{d}}\,f(Y)\,\pi(Y)^2\,dY.
\end{align*}

The formula for the Newton kernel requires more care. Let $\tilde{u}$ be the solution of the inhomogeneous Laplace equation on 
$\R^d$, then 
\begin{align*}
u(X)=\frac{\sum_{w\in W}\,\epsilon(w)\,\tilde{u}(w X)}{\pi(X)}
\end{align*}
solves the corresponding problem for $\Delta^W$.  We need however to show that $\lim_{X\to\infty}\,|u(X)|=0$.   It is useful to note that the function $\hat{u}(X)=\sum_{w\in W}\,\epsilon(w)\,\tilde{u}(w X)$ is skew-symmetric.

For $J\subseteq \{1,2,\dots,n\}$, les $A_J=\{x\in\R^n\colon \hbox{$|x_i|>1/2 $ for $i\in J$, $|x_i|<1 $ for $i\in J^c$}\}$.
Note that $\R^d$ is the union of the open sets $A_J$.
Now, on $A_J$ with $|J|\geq 1$ (so that $X\to\infty$),
\begin{align*}
\lefteqn{\lim_{X\to\infty}\,\left|\frac{\hat{u}(X)}{\pi(x_1,\dots,x_d)}\right|
=\lim_{(x_i)_{i\in J}\to\infty}\,\left|\frac{\hat{u}(X)}{\pi(x_1,\dots,x_d)}\right|
=\lim_{\exists i\in J,x_i\to0}\,\left|\frac{\hat{u}((1/x_i)_{i\in J},(x_i)_{i\in J^c})}{\pi((1/x_i)_{i\in J},(x_i)_{i\in J^c})}\right|}\\
&=\lim_{|x_i|<2, j\in J,|x_i|>1,i\in J^c,\exists i\in J,x_i\to0}\,\prod_{i\in J}\,|x_i|^{d-1}
\,\left|\frac{\hat{u}((1/x_i)_{i\in J},(x_i)_{i\in J^c})}{\pi((x_i)_{i\in J})\,\pi((x_i)_{i\in J^c})\,\prod_{i\in J,j\in J^c}(1-x_i\,x_j)}\right|.
\end{align*}

Observe that $\hat{u}((1/x_i)_{i\in J},(x_i)_{i\in J^c})$ is continuous since $\lim_{X\to\infty} \tilde{u}(X)=0$ and skew-symmetric in $(x_i)_{i\in J}$ and in $(x_i)_{i\in J^c}$.  We remark also that it is zero when $1-x_i\,x_j=0$, $i\in J$, $j\in J^c$.  Using 
{Lemma \ref{LLL} below, we can conclude that the term
\begin{align*}
\left|\frac{\hat{u}((1/x_i)_{i\in J},(x_i)_{i\in J^c})}{\pi((x_i)_{i\in J})\,\pi((x_i)_{i\in J^c})\,\prod_{i\in J,j\in J^c}(1-x_i\,x_j)}\right|
\end{align*}
is an analytic function so that} it remains bounded and that the limit is 0.
We are grateful to the anonymous referee for pointing out the need for additional justification in the Newton kernel case.
\end{proof}

\begin{lemma}\label{LLL}
Let $f:\R^2 \to \R$, $h:\R^2 \to \R$ two analytical functions such that $h^{-1}(\{0\}) \subset f^{-1}(\{0\})$. 
Suppose that for each $z_0=
(x_0,y_0) \in h^{-1}(\{0\})$, 
the order of $y_0$ as a zero of 
$h(x_0,\cdot)$ is one (i.e. $h(x_0,y)=(y-y_0)\tilde h(y), \tilde h$ analytical, $\tilde h(y_0)\not=0$).

Then there exists an analytical function $g:\R^2 \to \R$ such that $f=hg$.
\end{lemma}

\begin{proof}
We apply the Weierstrass division theorem (\cite[Th. 0.43(2)]{Denk},
\cite[Th.6.1.3(1)]{Krantz}).
For each $z_0=
(x_0,y_0) \in h^{-1}(\{0\})$
there exists a neighbourhood 
$V_{z_0}$ and analytical functions
$v_{z_0}(x,y)$ and $b_1(x)$ such that
$$ f(x,y)=h(x,y)v_{z_0}(x,y) + b_1(x).
$$
For all $(x,y)\in h^{-1}(0)\cap V_{z_0}$ the last equality gives
$0=0+b_1(x)$,  so that
$$
f(x,y)=h(x,y)v_{z_0}(x,y),
\qquad (x,y)\in V_{z_0}.
$$
An application of the principle of identity ends the proof.
\end{proof}

\begin{remark}
The properties of factorization of analytical functions of several real variables 
are not as straightforward as one might hope. For example, consider
$f(x,y)=y^3$ which is zero whenever $x^2+y^2=0$.  
However, it is not true that $f$ divided by $x^2+y^2$ is analytic or even defined. 
\end{remark}

For the root systems of type $A$, we obtain the following determinantal formula for the heat kernel on $M$. 
This formula may be also deduced from the formula for the transition function of the Dyson Brownian Motion, based on the Doob transform and Karlin-MacGregor formula, see Section \ref{DYSON}.

\begin{cor}\label{DetAd}
Consider the flat complex symmetric space $M$ with the root system $\Sigma=A_{d-1}$. Let $g_t(u,v)=\frac{1}{\sqrt{4\,\pi \,t}}\,e^{-|u-v|^2/4t}$ be the 1-dimensional classical heat kernel.
The heat kernel on 
$M$ is given by
\begin{align}\label{detAd}
p_t^W(X,Y)=\frac{1}{|W|\,\pi(X)\,\pi(Y)} \,\det\left(g_t(x_i,y_j)\right)
\end{align}
where $x_1$, \dots, $x_d$ are the coordinates of $X$ and $y_1$, \dots, $y_d$ are the coordinates of $Y$.
\end{cor}
\begin{proof}
Formula \eqref{detAd} follows from Theorem \ref{kernels}~(1) and the definition of determinant.
\end{proof}

\begin{remark}
In \cite{Grabiner}, Grabiner computes determinant formulas   for the transition probabilities of the Dyson Brownian motion in the Weyl chambers of $A_{n-1}$, $B_n,C_n$ and $D_n$.

Note that the alternating sum formula \eqref{alter} reduces to a determinant if and only if the kernel 
${\mathcal K}(X,Y)$  has a multiplicative form
$$
{\mathcal K}(X,Y)=\prod_{i=1}^d k(x_i, y_i).
$$
This holds true for the transition probabilities  of the Brownian Motion on $\R^d$ or, more generally, of any multidimensional stochastic process $\underline{X}(t)$ with independent identically distributed components 
$X_i(t)$.
\end{remark}

Let us resume the method of alternating sums, applied in the proof of Theorem \ref{kernels}. An Euclidean kernel ${\mathcal K}(X,Y)$ (heat kernel, potential kernel, Poisson kernel, \dots) for the Laplacian $\Delta^{\R^d}$ is transformed in the following way into the kernel ${\mathcal K}^W$ 
acting on $W$-invariant functions on $M$:
\begin{equation}\label{W_Method}
{\mathcal K}^W(X,Y)=\frac{1}{|W|\,\pi(X)\,\pi(Y)} \,\sum_{w\in W}\,{\epsilon(w)} {\mathcal K}(X,w\cdot Y).
\end{equation}

Formula \eqref {HeatComplexX} is immediate
from the explicit form of the heat kernel in Dunkl theory (refer to \cite{R0}) together with Proposition \ref{new spherical} below. The formulas \eqref{NewtonComplexX}-\eqref{detAd} are new. 

However, in the harmonic analysis of flat symmetric spaces of complex type, the alternating sum formula \eqref{spherical fn} for a spherical function on $M$ given below is well known (see \cite[Chap.{} IV, Proposition 4.8 and Chap.{} II, Theorem 5.35]{Helgason3}). Dunkl  had provided a proof for the root system $A_{n-1}$ in \cite{Du0} using a similar approach as ours.

\begin{prop}\label{new spherical}
Given $\lambda\in \a_\C^*$ (the dual of the complexification of $\a$), the spherical function $\psi_{\lambda}(X)$ on $M$ is given by the formula
\begin{equation}\label{spherical fn}
\psi_{\lambda}(X)=\frac{\pi(\rho)}{2^\gamma\pi(\lambda)\,\pi(X)}
\,\sum_{w\in W}\,{\epsilon(w)} e^{\langle \lambda, w\cdot X\rangle},
\end{equation}
where $\rho=\frac{1}{2} \,\sum_{\alpha\in\Sigma^+}\,m_\alpha \alpha=
\sum_{\alpha\in\Sigma^+} \,\alpha$ and $\gamma=|\Sigma^+|$ is the number of positive roots.
\end{prop}

We finish this section with a relation between the heat kernel and spherical functions which will be useful in stochastic applications of our results, see
Proposition \ref{K} and Corollary \ref{KDyson}.
Proposition \ref{25} is an immediate consequence of well-known
results by R\"osler in Dunkl theory (see for instance \cite[Lemma 4.5]{R0})
and of \cite{dJ}, ensuring that
heat kernel and spherical functions on
 flat complex symmetric spaces coincide with their Weyl group invariant analogues in
Dunkl analysis when the multiplicity $k = 1$.

\begin{prop}\label{25}
Let $M$ be a flat symmetric space of complex type. The following formula holds
\begin{equation}\label{heat-spherical}
p_t^W(X,Y)=\frac{1}{|W|\,2^d\,\pi^{d/2}\,\pi(\rho)} \,t^{-\frac{d}{2}-\gamma} \,e^{\frac{-|X|^2-|Y|^2}{4t}}\,\psi_X\left(\frac{Y}{2t}\right).
\end{equation}
\end{prop}

\begin{remark}
We provide here a simple explanation for the constant occurring in \eqref{heat-spherical}. From \eqref{HeatComplexX} and \eqref{spherical fn},
\begin{align*}
p_t^W(X,Y)
&=\frac{1}{|W|\,(4\,\pi\, t)^{d/2}\,\pi(X)\,\pi(Y)}\,\sum_{w\in W}\,{\epsilon(w)}{e^{-{\frac{|X-w\cdot Y|^{2}}{4\,t}}}}\\
&=\frac{1}{|W|\,2^d\,\pi^{d/2}\,\pi(\rho)} \,t^{-\frac{d}{2}-\gamma}
\,\frac{\pi(\rho)}{2^\gamma\,\pi(X)\,\pi\left(\frac{Y}{2t}\right)}\,\sum_{w\in W}\, \epsilon(w) \,e^{\frac{\langle X,w\cdot Y\rangle}{2t}}\\
&=
\frac{1}{|W|\,2^d\,\pi^{d/2}\,\pi(\rho)} \,t^{-\frac{d}{2}-\gamma} \,e^{\frac{-|X|^2-|Y|^2}{4t}}\,\psi_X\left(\frac{Y}{2t}\right).
\end{align*}
Note that the constants in \cite{R0} lead to the same constant as in  \eqref{heat-spherical} even though the functional $\rho$
is not used in the context of Dunkl theory. The same phenomenon
will appear for the constant for the Poisson kernel, see Remark 
\ref{normalization}.
\end{remark}

%%%%%%%%%%%%%%%%%%%%%%%%%%%%%%%%%%%%%%%%%%%%%%%%%%%%%%%%%%%%%

\section{Asymptotic behavior of the kernels}\label{asymp}
To simplify the notation, we will write $f\,\stackrel{Y_0}{\sim}\,g$ if $\lim_{X\to Y_0}\,\frac{f(X)}{g(X)}=1$. 

The main results of this Section are Theorems \ref{Poiss} and	\ref{Newt} which give asymptotics	of the Poisson and Newton kernels of the flat complex symmetric space $M$. In their proofs, we need some knowledge of Dunkl analysis on $\R^d$.

Consider $\R^d$ with a root system $\Sigma$. The basic information on the Dunkl analysis in this context can be found in \cite{Roesler}. Denote the Dunkl Laplacian by $\Delta_k$ and the intertwining operator by $V_k$.
	
Recall now the formula of Dunkl (\cite{Du1,DuX}) for the Dunkl-Poisson kernel of the unit open ball $B=B(0,1)$.

\begin{equation}\label{PoissonDunkl}
P_k(X,Y)=\frac{2^{2\,\gamma}\,(d/2)_\gamma}{\pi(\rho)\,|W|\,w_d}\,V_k\left[\frac{1-|X|^2}{(1-2\langle X,\cdot\rangle +|X|^2)^{\gamma+
d/2}}\right] (Y),\ \ X\in B,~ Y\in\partial B,~ \gamma = \sum_{\alpha\in \Sigma}\,k_\alpha.
\end{equation}

The constant in \eqref{PoissonDunkl} is different from the one given in \cite{Du1,DuX}.  Our constant is explained below in Remark \ref{normalization}.

The flat complex symmetric spaces $M$ correspond to the formula
\eqref{PoissonDunkl} in the $W$-invariant case and with $k_\alpha=1$. Then $\gamma=|\Sigma_+|$ expresses the number of positive roots. 
	
A formula for the Dunkl-Newton kernel $N_k(X,Y)$,   analogous to \eqref{PoissonDunkl}, was proven in \cite{Gallardo2}.
%%%%%%%%%%%%%%%%%%%%%%%%%%%%%%%%%%%%%%%%%%%%%%%%%%%%%%%
%%%%%%%%%%%%%%%%%%%% POISSON
%%%%%%%%%%%%%%%%%%%%%%%%%%%%%%%%%%%%%%%%%%%%%%%%%%%%

\subsection{Poisson kernel of the flat complex symmetric space}

The following technical results will prove useful further on.

\begin{lemma}\label{tech}
\begin{align*}
\partial(\pi)\,|X|^{-d}
&=2^\gamma\,\prod_{k=0}^{\gamma-1}\,(-d/2-k)\,\pi(X)\,|X|^{-d-2\,\gamma}\\
\partial(\pi)\,\log |X|
&=(-2)^{\gamma-1}\,(\gamma-1)!\,\pi(X)\,|X|^{-2\,\gamma}.
\end{align*}
\end{lemma}

\begin{proof}
We see easily that $|X|^{d+2\,\gamma}\,\partial(\pi)\,|X|^{-d}$ is a skew polynomial of degree at most $\gamma$. It must therefore be a constant multiple of $\pi(X)$. Note from the definition of $\partial(\pi)$ that 
\begin{align*}
\partial(\pi)\,f(X)=\prod_{\alpha>0}\,\left.\frac{\partial~}{\partial t_\alpha}\right|_{t_\alpha=0}\,f(X+\sum_{\alpha>0}\,t_\alpha\,H_\alpha) 
\end{align*}
where $H_\alpha$ is defined by the relation $\langle X,H_\alpha\rangle=\alpha(X)$ for $X\in\a$. Hence, 
\begin{align*}
\partial(\pi)\,|X|^{-d}=\prod_{\alpha>0}\,\left.\frac{\partial~}{\partial t_\alpha}\right|_{t_\alpha=0}\,\langle X+\sum_{\alpha>0}\,t_\alpha\,H_\alpha,X+\sum_{\alpha>0}\,t_\alpha\,H_\alpha\rangle^{-d/2}.
\end{align*}

After applying the operators $\left.\frac{\partial~}{\partial t_\alpha}\right|_{t_\alpha=0}$, we will be left with 
the term 
\begin{align*}
(-d/2)\,(-d/2-1)\,\cdots\,(d/2-(\gamma-1))\,\langle X,X\rangle^{-d/2-\gamma}\,\prod_{\alpha>0}\,(2\,\alpha(X))
\end{align*}
and other terms which do not have the right form.  This tells us that desired constant is $2^\gamma\,\prod_{k=0}^{\gamma-1}\,(-d/2-k)$.

A similar reasoning applies for the computation of $\partial(\pi)\,\log |X|$.
\end{proof}

\begin{prop}\label{tech2}
Let $T(X,Y)=\frac{1}{\pi(X)\,\pi(Y)} \,\sum_{w\in W}\,\frac{\epsilon(w)}{|X-w\cdot Y|^{d}}$. 
Then $T(0,Y)=\frac{2^{2\,\gamma}\,(d/2)_\gamma}{\pi(\rho)}\,|Y|^{-d-2\,\gamma}$.
\end{prop}

\begin{proof}
Note first that $\partial(\pi)_X\,|X-Y|^{-d}=2^\gamma\,\prod_{k=0}^{\gamma-1}\,(-d/2-k)\,\pi(X-Y)\,|X-Y|^{-d-2\,\gamma}$. Consider 
$B(X,Y)=\pi(X)\,T(X,Y)=\frac{1}{\pi(Y)} \,\sum_{w\in W}\,\epsilon(w)\,|X-w\cdot Y|^{-d}$.
We apply the differential operator $\left.\partial(\pi)\right|_{X=0}$ to $B$. We find
\begin{align*}
\partial(\pi)(\pi)\,T(0,Y)&=2^\gamma\,\prod_{k=0}^{\gamma-1}\,(-d/2-k)
\,\frac{1}{\pi(Y)} \,\left.\sum_{w\in W}\,\epsilon(w)\,\pi(X-w\cdot Y)\,|X-w\cdot Y|^{-d-2\,\gamma}\right|_{X=0}\\
&=(-1)^\gamma\,2^\gamma\,\prod_{k=0}^{\gamma-1}\,(-d/2-k)\,|W|\,|Y|^{-d-2\,\gamma}.
\end{align*}

Finally,
\begin{align*}
T(0,Y)
&=\frac{(-1)^\gamma\,2^\gamma\,\prod_{k=0}^{\gamma-1}\,(-d/2-k)\,|W|}{\partial(\pi)(\pi)} \,|Y|^{-d-2\,\gamma}
=\frac{(-1)^\gamma\,2^\gamma\,\prod_{k=0}^{\gamma-1}\,(-d/2-k)\,|W|}{\pi(\rho)\,|W|/2^\gamma} \,|Y|^{-d-2\,\gamma}\\
&=\frac{2^{2\,\gamma}\,(d/2)_\gamma}{\pi(\rho)}\,|Y|^{-d-2\,\gamma}.
\end{align*}
\end{proof}

\begin{cor}\label{at0}
We have
\begin{align*}
P^W(0,Y)&=\frac{2^{2\,\gamma}\,(d/2)_\gamma}{\pi(\rho)\,|W|\,w_d}\\
N^W(0,Y)&=\frac{-2^{2\,\gamma-1}\,(\gamma-1)!}{
2\,\pi\,|W|\,\pi(\rho)}
\,|Y|^{-2\,\gamma}\ \hbox{if $d=2$}\\
N^W(0,Y)&=\frac{2^{2\,\gamma}\,((d-2)/2)_\gamma}{|W|\,(2-d)\,w_d\,\pi(\rho)}\,|Y|^{2-d-2\,\gamma}\ \hbox{if $d\geq 3$}
\end{align*}
\end{cor}

\begin{prop}\label{PW}
The Poisson kernel of the unit ball on the flat complex symmetric space $M$ is given by
\begin{align}\label{PoissonDunklW}
P^W(X,Y)
&=\frac{2^{2\,\gamma}\,(d/2)_\gamma}{\pi(\rho)\,|W|\,w_d}\,\mathcal{A}^*\left(\frac{1-|X|^2}{(1-2\,\langle X,\cdot\rangle+| X|^2)^{\gamma+d/2}}\right)(Y),
\end{align}
where $\mathcal{A}^*$ denotes the dual Abel transform on $M$.
\end{prop}

Proposition \ref{PW} will be essential to establish \eqref{firstTerm} in Theorem \ref{Poiss}.
Recall that the dual of the Abel transform can be defined by the equation
\begin{align*}
\mathcal{A}^*(f)(X)=\int_K\,f(\pi_\a(\Ad(k)\,X))\,dk
\end{align*}
where, as before, $\pi_\a$ is the orthogonal projection from $\p$ to $\a$ with respect to the Killing form. Note in particular that 
$\mathcal{A}^*(e^{\lambda(\cdot)})(X)=\psi_\lambda(X)$. Note also {(see \cite[Ch.{} IV, Theorem 10.11]{Helgason3})} that unless $C(X)$ reduces to $\{X\}$, there exists a density $K(H,X)$ such that 
\begin{align*}
\mathcal{A}^*(f)(X)=\int_{C(X)}\,f(H)\,K(H,X)\,dH.
\end{align*}

\begin{proof}[Proof of Proposition \ref{PW}]	
It should be noted that for Weyl-invariant $f$, $\mathcal{A}^*(f)=V_k(f)$ (refer to \cite{dJ}). Since the argument of $\mathcal{A}^*$ in \eqref{PoissonDunklW} is not Weyl-invariant, some proof is needed.
Let $K(Z,Y)$ be the kernel of the dual Abel transform. Using \eqref{PoissonDunkl}, we have
\begin{align*}
P^W(X,Y)&=\frac{\sum_{w,w_0\in W}\,P_k(w\cdot X,w_0\cdot Y)}{|W|^2}\qquad \hbox{(with $k=1$)}\\
&=\frac{2^{2\,\gamma}\,(d/2)_\gamma}{\pi(\rho)\,|W|^3\,w_d}\,\sum_{w,w_0\in W} \,\int_{C(w_0\cdot Y)} \,\frac{1-|w\cdot X|^2}{(1-2\,\langle w\cdot X,Z\rangle+|w\cdot X|^2)^{\gamma+d/2}}\,d\mu_{w_0\cdot Y}(Z)\\
&=\frac{2^{2\,\gamma}\,(d/2)_\gamma}{\pi(\rho)\,|W|^3\,w_d}\,(1-|X|^2)\,\int_{C(Y)} \,\sum_{w,w_0\in W}\,\frac{1}{(1-2\,\langle w\cdot X,Z\rangle+| X|^2)^{\gamma+d/2}} \,d\mu_{Y}(w_0^{-1}\cdot Z)\\
&=\frac{2^{2\,\gamma}\,(d/2)_\gamma}{\pi(\rho)\,|W|^3\,w_d} \,(1-|X|^2)\,\int_{C(Y)} \,\overbrace{\sum_{w,w_0\in W}\,\frac{1}{(1-2\,\langle w\cdot X,w_0\cdot Z\rangle+| X|^2)^{\gamma+d/2}}}^{\hbox{Weyl-invariant}}	 \,d\mu_{Y}(Z)\\
&=\frac{2^{2\,\gamma}\,(d/2)_\gamma}{\pi(\rho)\,|W|^3\,w_d}\,(1-|X|^2)\,\int_{C(Y)} \,\sum_{w,w_0\in W}\,\frac{1}{(1-2\,\langle w\cdot X,w_0\cdot Z\rangle+| X|^2)^{\gamma+d/2}} \,K(Z,Y)\,dZ\\
&=\frac{2^{2\,\gamma}\,(d/2)_\gamma}{\pi(\rho)\,|W|^3\,w_d}\,(1-|X|^2)\,\int_{C(Y)} \,\sum_{w,w_0\in W}\,\frac{1}{(1-2\,\langle X,w^{-1}w_0\cdot Z\rangle+| X|^2)^{\gamma+d/2}} \,K(Z,Y)\,dZ\\
&=\frac{2^{2\,\gamma}\,(d/2)_\gamma}{\pi(\rho)\,|W|^3\,w_d}\,(1-|X|^2)\,\int_{C(Y)} \,\sum_{w,w_0\in W}\,\frac{1}{(1-2\,\langle X,Z\rangle+| X|^2)^{\gamma+d/2}} \,K(w_0^{-1}\,w\,Z,Y)\,dZ\\
&=\frac{2^{2\,\gamma}\,(d/2)_\gamma}{\pi(\rho)\,|W|\,w_d}\,(1-|X|^2)\,\int_{C(Y)} \,\frac{1}{(1-2\,\langle X,Z\rangle+| X|^2)^{\gamma+d/2}} \,K(Z,Y)\,dZ.
\end{align*} 
\end{proof}

\begin{remark}\label{normalization}
Note that our normalizing constant is different from what is found in \cite{Du1,DuX}.  We explain here how they correspond in the complex case.  In \cite{Du1}, the Poisson kernel $P^W(X,Y)$ is normalized in the following manner: 
\begin{align*}
u(X)=c_d'\,\int_{\partial B(0,1)}\,f(y)\,P^W(X,Y)\,\pi(Y)^2\,\frac{dY}{w_d}
\end{align*}
where $c_d'$ is such that
\begin{align*}
1=c_d'\,\int_{\partial B(0,1)}\,\pi(Y)^2\,\frac{dY}{w_d}
\end{align*}
where, reading through \cite[Page 1215]{Du1},
\begin{align*}
c_d'=2^\gamma\,(d/2)_\gamma\,\prod_{\alpha>0}\,\left(\frac{|\alpha|^2}{2}\,(\langle\alpha,\rho\rangle/|\alpha|^2+1)\right)^{-1}.
\end{align*}
Our different normalizations come down to the equality
\begin{align*}
\frac{2^{2\,\gamma}\,(d/2)_\gamma}{\pi(\rho)\,|W|\,w_d}=\frac{c_d'}{w_d}
\end{align*}
which gives the interesting equality
\begin{align*}
\frac{\pi(\rho)\,|W|}{2^\gamma}=\prod_{\alpha>0}\,\left(\frac{|\alpha|^2}{2}\,(\langle\alpha,\rho\rangle/|\alpha|^2+1)\right).
\end{align*}
This equality is easily verified directly for the classical Lie algebras and for $\g_2$ (the other exceptional Lie algebras require more work).
It should be noted that in \cite{Du1}, Dunkl used the notation $\nu_h$ instead of $\rho$ but refers to the fact that Opdam uses $\rho$ in \cite{Opdam}.
\end{remark}

\begin{cor}\label{NewtonAstar}
The Newton kernel of the flat complex symmetric space $M$ is given by
\begin{align*}
N^W(X,Y)
&=\frac{2^{2\,\gamma}\,((d-2)/2)_\gamma}{|W|\,(2-d)\,w_d\,\pi(\rho)}\,\mathcal{A}^*\left({(|Y|^2-2\,\langle X,\cdot\rangle+| X|^2)^{(2-d-2\,\gamma)/2}} \right)(Y).
\end{align*}
\end{cor}

\begin{proof}
We apply the same computations as for the Poisson kernel to formula \cite[(6.1)]{Gallardo2}
(the constant has been adjusted to follow our conventions as per Remark \ref{normalization}).
\end{proof}

We now start to study the asymptotic behavior of the Poisson kernel $P^W(X,Y)$. Let us introduce some notations.
We define 
\begin{align*}
R(X,Y)= \sum_{w\in W}\,\frac{\epsilon(w)}{|X-w\cdot Y|^{d}}\ \ {\rm and}\ \ T(X,Y)=\frac{R(X,Y)}{\pi(X)\pi(Y)}
\end{align*}
and therefore, 
\begin{align*}
P^W(X,Y)=\frac{1-|X|^2}{|W|\,w_d}\,T(X,Y).
\end{align*}
%It is sufficient to study the Poisson kernel %$P^W(X,Y)=(1-|X|^2)T(X,Y)$ and the 
%the functions $T$ and $R$ for $X,Y\in\overline{\a+}$.
The function $R(X,Y)$ is defined for $X,Y\in{\a}$ such that $X\not\in W\cdot Y=\{w\cdot Y|\ w\in W\}$. We will denote this domain by
\begin{align*}
D:=\{(X,Y)\in\a^2|\ X\not\in W\cdot Y\}
\end{align*} 
The function $T(X,Y)$ is, for now, defined for non-singular $X$, $Y\in\a$ (\emph{i.e.} such no nonzero root vanish on $X$ or on $Y$) such that $X\not\in WY$. We will see in Proposition
\ref{PropertiesRT} that the function $T(X,Y)$ extends by continuity to an analytic function on the domain $D$.

Studying the properties of $P^W(X,Y)$ is equivalent to studying the properties of $T(X,Y)$ and $R(X,Y)$. We will give some of them in Proposition
\ref{PropertiesRT}. We start by introducing two auxiliary results.

\begin{lemma}\label{prep}
Assume $a_1$, \dots, $a_d$ are not all 0 and let $U$ be an open set. 
Let $q$ be an analytic function on $U$ which is 0 whenever $\sum_{k=1}^d\,a_k\,x_k=0$. Then $q(X)=\left(\sum_{k=1}^d\,a_k\,x_k\right) \,r(X)$ where $r$ is a analytic function on $U$.
\end{lemma}

\begin{proof}
The Lemma follows from Lemma \ref{LLL}. We give here an elementary proof.

Using a change of variable, we can assume that $a_1=1$ and $a_i=0$ for $i>1$. It is also enough to show that for every $X_0=(b_1,\dots,b_d)\in U$, there exists $\epsilon>0$ such that the result holds in the ball $B(X_0,\epsilon)$. If $X_0\not=0$, then pick $\epsilon>0$ small enough so that $(x_1,\dots,x_d)\in B(X_0, \epsilon)$ implies $x_1\not=0$. Then we can pick $r(X)=q(X)/x_1$.

Suppose now that $b_1=0$. We then have
$\displaystyle q(x_1,\dots,x_d)=x_1\,\overbrace{\int_0^1\,\frac{\partial~}{\partial x_1}q(t\,x_1,x_2,\dots,x_d)\,dt}^{r(X)}$ for $(x_1,\dots,x_d)\in B(X_0,\epsilon)\subset U$.
\end{proof}

\begin{prop}\label{AnalyticFactor}
Let $p(X)=\prod_{i=1}^d\,\langle\alpha_i,X\rangle$ where no $\alpha_i$'s is a multiple of another $\alpha_j$ and let $U$ be an open set. If $q$ is an analytic function on $U$ which is 0 whenever $\alpha_i(X)=0$ for some $i$ then $q(X)=p(X)\,r(X)$ where $r$ is an analytic function on $U$.
\end{prop}

\begin{proof}
We use induction on $n$. Lemma \ref{prep} shows that the result is true for $n=1$. Assume it is true for $n-1$, $n\geq2$ and write $q(X)=
\prod_{i=1}^{n-1}\,\langle\alpha_i,X\rangle\,r(X)$. Since $q(X)=0$ when $\langle\alpha_n,X\rangle=0$, we conclude that
$r(X)=0$ on the set $\{X|\langle\alpha_n,X\rangle=0~\hbox{and}~\langle\alpha_i,X\rangle\not=0,~i<n\}$. By continuity, we deduce that $r(X)=0$ when 
$\langle\alpha_n,X\rangle=0$ and, using Lemma \ref{prep} once more, we can conclude.
\end{proof}

\begin{remark}
We thank the referee for pointing out that
Lemma \ref{prep} and 
Proposition \ref{AnalyticFactor} can be also be proven with the help of the Weierstrass division theorem, via Lemma \ref{LLL}.
\end{remark}

\begin{prop}\label{PropertiesRT}
~
\begin{enumerate}
\item (Symmetry in $X$ and $Y$) $R(X,Y)=R(Y,X)$ and $T(X,Y)=T(Y,X)$.

%%%%%%%%%%%%%%%%%%%%%%%%%%%%%%%%%%%%%%%%%%%%%%%%%%%%%%%%%%%%%%%%%%
\item (Skew-symmetry) $R(w_0 X,Y)=\epsilon(w_0)\, R(X,Y)$ and $R(X, w_0 Y)=\epsilon(w_0)\, R(X,Y)$.
%%%%%%%%%%%%%%%%%%%%%%%%%%%%%%%%%%%%%%%%%%%%%%%%%%%%%%%%%%%%%%%%%%%%%%
\item (Nullity of $R$ on singular arguments) $R(X,Y)$ is zero whenever at least one of $X$ or $Y$ is singular.
%%%%%%%%%%%%%%%%%%%%%%%%%%%%%%%%%%%%%%%%%%%%%%%%%%%%%%%%%%%%%%%%
\item (analytic factorization of $R$, analytic extension of $T$ to $D$.)
There exists a function $f$ analytic on $D$ such that
$R(X,Y)=\pi(X)\pi(Y) f(X,Y)$ on $D$. Equivalently, the function $T$ extends to an analytic function on $D$.
%%%%%%%%%%%%%%%%%%%%%%%%%%%%%%%%%%%%%%%%%%%%%%%%%%%%%%%%%%%%%%%%%%%%%%%%
\item (Non-nullity of $T$ and $P^W$) When $X\in B$ and $Y\in\partial B$ then $T(X,Y)>0$ and $P^W(X,Y)>0$.
\end{enumerate}
\end{prop}

\begin{proof}

The proof of (1) and (2) is straightforward.
\begin{itemize}
\item[(3)] Suppose $\alpha(Y)=0$. We use Property 2 and $\epsilon(\sigma_\alpha)=-1$ where $\sigma_\alpha$ is the reflection
with respect to the hyperplane $\{\alpha=0\}$. Since $R(X,Y)/\pi(Y)$ is analytic, the statement follows.

\item[(4)] This follows from Proposition \ref{AnalyticFactor}.
	
\item[(5)]	This follows from Proposition \ref{PW}.
The dual Abel integral transform  of a strictly positive function is strictly positive.
\end{itemize}

\end{proof}

\begin{theorem}\label{Poiss}
Let $Y_0\in \partial B$, $\Sigma'=\{\alpha\in\Sigma|\ \alpha(Y_0)=0\}$ and $\Sigma'_+=\Sigma'\cap\Sigma^+$.
Then 
\begin{equation}\label{asymptPoisson}
P^W(X,Y_0)\stackrel{Y_0}{\sim} \frac{2^{2\,\gamma'}\,(d/2)_{\gamma'}}{|W|\,w_d\,\pi'(\rho')\,(\pi''(Y_0))^2}\,\frac{1-|X|^2}{|X-Y_0|^{2\gamma'+d}}
\end{equation}
where $\gamma'=|\Sigma_+'|$ is the number of positive roots annihilating $Y_0$, {$\pi'(Y)=\prod_{\alpha\in \Sigma_+'}\,\langle\alpha,Y\rangle$} and $\pi''(Y)=\prod_{\alpha\in\Sigma_+\setminus \Sigma_+'}\,\langle\alpha,Y\rangle$.
\end{theorem}
\begin{proof}
Let $W'=\{w\in W| w\cdot Y=Y\}$.  In this proof, we consider $X\in V=B(Y_0,\epsilon)$
with $\epsilon>0$ fixed and chosen in such a way that
\begin{align*}
\alpha(\bar{V})\subset (0,\infty) \ {\rm for}\ \alpha\in \Sigma_+ \setminus \Sigma_+' 
\ \ {\rm and } \ \ w\, V\cap V=\emptyset\ \hbox{for every $w\in W\setminus W'$}.
\end{align*}

Using Theorem \ref{kernels}, we have
\begin{align*}
P^W(X,Y)=\frac{1}{|W|\,w_d}\,\frac{1-|X|^2}{\pi(X)\,\pi(Y)}\,\sum_{w\in W}\,\frac{\epsilon(w)}{|X-w\cdot Y|^{d}}.
\end{align*} 

We consider $X\in V\setminus \{Y_0\}$ and we deal with
\begin{align}
T(X,Y_0){=} \frac{|W|\,w_d}{1-|X|^2} \,P^W(X,Y_0)= \frac{1}{\pi(X)\,\pi(Y_0)}
\,\sum_{w\in W}\,\frac{\epsilon(w)}{|X-w\cdot Y_0|^{d}}=\frac{R(X,Y_0)}{\pi(X)\,\pi(Y_0)}.\label{TY0}
\end{align}
By Proposition \ref{PropertiesRT} applied to the root systems $\Sigma$ and $\Sigma'$, all the expressions in \eqref{TY0} are well defined for $X\in V\setminus \{Y_0\}$, if needed in the limit sense.
 
We decompose the sum $\sum_{w\in W}$
into two terms, the first being the sum over the subgroup $W'=\{w\in W| w\cdot Y_0=Y_0\}$ which is the Weyl group of the root subsystem 
$\Sigma'$. We obtain
\begin{align*}
T(X,Y_0)=\frac{\sum_{w\in W} {\epsilon(w)}|X-w\cdot Y_0|^{-d}}{\pi(X)\,\pi(Y_0)}=
\frac{\sum_{w\in W'} {\epsilon(w)}{|X-w\cdot Y_0|^{-d}}}{\pi(X)\,\pi(Y_0)}+
\frac{ \sum_{w\in W\setminus W'} {\epsilon(w)}{|X-w\cdot Y_0|^{-d}}}{\pi(X)\,\pi(Y_0)}.
\end{align*}
By Proposition \ref{PropertiesRT}, all the expressions in the last formula are well defined for $X\in V\setminus \{Y_0\}$, if needed in the limit sense.
Denote
\begin{align*}
T_1(X,Y_0)= \frac{\sum_{w\in W'} {\epsilon(w)}{|X-w\cdot Y_0|^{-d}}}{\pi(X)\,\pi(Y_0)}\ \ {\rm and}\ T_2(X,Y_0)=
\frac{ \sum_{w\in W\setminus W'} {\epsilon(w)}{|X-w\cdot Y_0|^{-d}}}{\pi(X)\,\pi(Y_0)}.
\end{align*} 
Let $\pi'(X)=\prod_{\alpha\in\Sigma_+'} \alpha(X)$ and
$\pi''(X)=\prod_{\alpha\in \Sigma_+\setminus \Sigma_+'} \alpha(X)$.
Observe that by Theorem \ref{kernels},
\begin{align*}
\pi''(X)\,\pi''(Y_0)\,T_1(X,Y_0)= \frac{\sum_{w\in W'}\,{\epsilon(w)}{|X-w\cdot Y_0|^{-d}}}{\pi'(X)\,\pi'(Y_0)}=\frac{|W'|\,w_d}{1-|X|^2} \,P^{W'}(X,Y_0)
\end{align*}
where $P^{W'}(X,Y)$ is the Poisson kernel for the flat symmetric space $(\R^d, \Sigma')$ corresponding to the complex root system $\Sigma'$. The convex hull
$C'(Y_0)=\hbox{conv}({W'\,Y_0})= \{Y_0\}$, so by Proposition \ref{PW} and the properties of $\mathcal{A}^*$, 
\begin{align}
\frac{1}{1-|X|^2}\,P^{W'}(X,Y_0)&=\frac{2^{2\,\gamma'}\,(d/2)_{\gamma'}}{\pi(\rho')\,|W'|\,w_d}
\,\int _{C(Y_0)}\,\frac{1}{(1-2\,\langle X, Z\rangle+| X|^2)^{\gamma'+d/2}} \delta_{\{Y_0\}}(dZ) \nonumber \\
&= \frac{2^{2\,\gamma'}\,(d/2)_{\gamma'}}{\pi'(\rho')\,|W'|\,w_d}\,\frac{1}{|X-Y_0|^{2\gamma'+d}} \label{firstTerm}
\end{align}
where $X\in B\cap V$.
 
We now prove that the function $X\mapsto T_2(X,Y_0)$ is bounded on $V$, which, together with \eqref{firstTerm}, will conclude the proof.
We denote by
\begin{align*}N(X,Y)={ \sum_{w\in W\setminus W'} {\epsilon(w)}{|X-w\cdot Y|^{-d}}}\end{align*}
the numerator of
$T_2$. Observe that $N(X,Y)$ is an analytic function on $V\times V$. The function 
\begin{align*}
T_2(X,Y)=
\frac{ \sum_{w\in W\setminus W'} {\epsilon(w)}{|X-w\cdot Y|^{-d}}}{\pi(X)\,\pi(Y)}
\end{align*}
is well defined and analytic for $(X,Y)\in V\times V\setminus D$ with $D=\{(X,Y)\in\a\times\a\colon X=Y\}$, 
since $T(X,Y)$ and $T_1(X,Y)$ have these properties by Proposition \ref{PropertiesRT} and 
$T_2=T-T_1$.

This implies that if $ X'\in V$ or $ Y'\in V$ are singular (\emph{i.e.} $\alpha(X')=0$ or $\alpha(Y')=0$ for some $\alpha\in \Sigma_+'$) and $X'\not=Y'$ then the numerator $N(X',Y')=0$ since otherwise the limit $N(X,Y)/\pi(X)\pi(Y)$ could not exist
when $(X,Y) \to (X',Y')$.

We deduce that if $ X'\in V$ or $ Y'\in V$ and $\alpha(X')=0$ or $\alpha(Y')=0$ for some $\alpha\in \Sigma_+'$) then $N(X',Y')=0$. 
This is also true for $X'=Y'$ since such points are
limits when $t$ tends to 1 of  $(t\,X',Y')$ with singular $t\,X'\not=Y'$
and $N(t\,X',Y')$ converges to $ N(X',Y')$. 

By Proposition \ref{AnalyticFactor}, there exists a function $F(X,Y)$ analytic on $V\times V$ such that 
\begin{align*}
N(X,Y)=\pi'(X)\pi'(Y) F(X,Y),\ \ \ \ X,Y\in V
\end{align*}
and, finally, 
\begin{align*}
T_2(X,Y)=\frac{F(X,Y)}{\pi''(X)\pi''(Y)},\ \ \ \ X,Y\in V
\end{align*}
(we have $\min_{X\in \bar{V}}\,\pi''(X)>0$ since $\pi''(\bar{V})\subset (0,\infty)$). 
In particular, the function $X\mapsto T_2(X,Y_0)$ is bounded on $V$.
\end{proof}

%%%%%%%%%%%%%%%%%%%%%%%%%%%%%%%%%%%%%%%%%%%%%%%%%%%%%%%%%%%%%%%%%%%%%%%%
\begin{remark}
For the asymptotic properties of $P^W$, besides the alternating sum formula, the approach via the Dunkl 
formula
\eqref{PoissonDunkl} and dual Abel transform, {\it i.e.} Proposition \ref{PW} is needed.
We use it to compute the leading term $T_1(X,Y)$ in $T(X,Y)$.
\end{remark}

%%%%%%%%%%%%%%%%%%%%%%%%%%%%%%%%%%%%%%%%%%%%%%%%%%%%%%%%%%%%
%%%%%%%%%%%%%%%%%%%%%%%%%%%%%%%%%%%%%%%%%%%%%%%%%%%%%%%%%%%%
%%%%%%%%%%%%%%%%%%%%%% NEWTON %%%%%%%%%%%%%%%%%%%%%%%%%%%%%
%%%%%%%%%%%%%%%%%%%%%%%%%%%%%%%%%%%%%%%%%%%%%%%%%%%%%%%%%%
\subsection{Asymptotic behavior of the Newton kernel on flat complex symmetric spaces}

Using the same approach as in the proof of Theorem \ref{Poiss} together with Corollary \ref{NewtonAstar}, we conclude that 
\begin{theorem}\label{Newt}
Let $Y_0\in \overline{\a^+}$. 
If $d=2$ and $\alpha$, $\beta$ are the simple roots then
\begin{align*}
N^W(X,0)&=\frac{-2^{2\,\gamma-1}\,(\gamma-1)!}{
2\,\pi\,|W|\,\pi(\rho)}
\,|X|^{-2\,\gamma}\ \hbox{(case $Y_0=0$)},\\
N^W(X,Y_0)&\stackrel{Y_0}{\sim} \frac{-2^{2\,\gamma'-1}\,(\gamma'-1)!}{2\,\pi\,|W|\,\pi''(Y_0)^2
\,{\langle\alpha,\alpha\rangle}} \,|X-Y_0|^{-2}\ \hbox{where $Y_0\not=0$, $\alpha(Y_0)\not=0$ and $\beta(Y_0)=0$}.
\end{align*}

If $d\geq 3$
\begin{equation}\label{asymptNewton}
N^W(X,Y_0)\stackrel{Y_0}{\sim} \frac{2^{2\,\gamma'}\,((d-2)/2)_{\gamma'}}{|W|\,(2-d)\,w_d\,\pi'(\rho')\,(\pi''(Y_0))^2}\,\frac{1}{|X-Y_0|^{2\gamma'+d-2}}.
\end{equation}

Here $\gamma'=|\Sigma_+'|$ is the number of positive roots annihilating $Y_0$ and $\pi''(Y)=\prod_{\alpha\in\Sigma_+\setminus \Sigma_+'}\,\langle\alpha,Y\rangle$.
\end{theorem}

 \begin{remark}
 In the paper \cite{PGTLPS}  exact estimates
 of the Poisson and Newton kernels $P^W$
 and $N^W$ were proven complementing the results of Theorem \ref{Poiss} and Theorem \ref{Newt}. For the Poisson kernel it is proven that
\begin{align*}
P^W(X,Y)\asymp\frac{P^{\R^d}(X,Y)}{\prod_{\alpha\in \Sigma^+}|X-\sigma_\alpha Y|^{2}}
\end{align*}
where $\sigma_\alpha$ is the symmetry with respect to
 the hyperplane perpendicular to $\alpha$.
 \end{remark}

%%%%%%%%%%%%%%%%%%%%%%%%%%%%%%%%%%%%%%%%%%%%%%%%%%%%%%%%%%%
%%%%%%%%%%%%%%%%%%%%%%%%%%%%%%%%%%%%%%%%%%%%%%%%%%%%%%%%%%
%%%%%%%%%%%%%%%%%%%% SPHERICAL FUNCTIONS %%%%%%%%%%%%%%%%%%
%%%%%%%%%%%%%%%%%%%%%%%%%%%%%%%%%%%%%%%%%%%%%%%%%%%%%%%%%%
%%%%%%%%%%%%%%%%%%%%%%%%%%%%%%%%%%%%%%%%%%%%%%%%%%%%%%%%%%%

\section{Asymptotic behavior of spherical functions on flat complex symmetric spaces}\label{spherical}

In this section we consider spherical functions on $M$, satisfying the formula
\begin{equation}\label{sph fn_notHelgason}
\psi_{\lambda}(Y)=\frac{\pi(\rho)}{2^\gamma\pi(\lambda)\,\pi(Y)}
\,\sum_{w\in W}\,{\epsilon(w)} e^{\langle \lambda, w\cdot Y\rangle},
\ \ \lambda\in \a^\C,\ Y\in\a^{\R}.
\end{equation}
Note that our notation is different from that of Helgason
(in his notation the function given by \eqref{sph fn_notHelgason} is denoted $\psi_{-i\,\lambda}$). 

The following technical lemma will prove useful later in this section.

\begin{lemma}\label{strate}
Suppose $G_1$ and $G_2$ are subgroups of the finite group $G$. Then $|G_1\,G_2|\,|G_1\cap G_2|=|G_1|\,|G_2|$.
\end{lemma}

\begin{proof}
The group $G_1\times G_2$ acts on the set $G_1\,G_2\subset G$ via $(g_1,g_2)(g)=g_1\,g\,g_2^{-1}$. Clearly the action is transitive. The stabilizer of $e\in G_1\,G_2$ ($e$ being the identity) is easily seen to be isomorphic to $G_1\cap G_2$. The orbit-stabilizer theorem (\cite[Theorem 5.8]{Rose}) implies then that $|G_1\,G_2|\,|G_1\cap G_2|=|G_1|\,|G_2|$.
\end{proof}

We introduce here some notation. If $X\in\a$, we denote by $ \Sigma^+_X$ the positive root system $ \Sigma^+_X=\{\alpha\in\Sigma^+\colon \alpha(X)=0\}$ and by $W_X$ the Weyl group generated by the symmetries $s_\alpha$ with $\alpha\in \Sigma^+_X$ (consequently, $W_X=\{w\in W\colon w\cdot X=X\}$). We also write $\pi_X(Y)=\prod_{\alpha\in \Sigma^+_X} \alpha(Y)$ and $c_X=\partial(\pi_X)(\pi_X)$ (this derivative is constant on $\a$).

For $X\in\a$ we define the polynomial $\pi'_{X}(Y)$ by $\pi(Y)=\pi_{X}(Y)\pi'_{X}(Y)$. Denote 
\begin{align*}
W(\lambda_0,Y_0) =\{w\in W\colon \langle \lambda_0, w\cdot Y_0\rangle= \langle \lambda_0, Y_0\rangle \}.
\end{align*}

\begin{remark}
We conjecture that the property $W(\lambda_0,Y_0)=W_{\lambda_0} \,W_{Y_0}$ is valid for all root systems. In Appendix \ref{KM}, we provide a series of proofs that cover all cases except for the exceptional root systems of type $E$. We also point out that if one of $\lambda_0$ or $Y_0$ is regular then this property is also verified, see \cite{Helgason1}.
\end{remark}

Denote the Weyl subgroup $W_{\lambda_0,Y_0}=W_{\lambda_0}\cap W_{Y_0}=\{w\in W\colon w\cdot \lambda_0=\lambda_0 \ {\rm and}\ w\cdot Y=Y\}$. The group  
$W_{\lambda_0,Y_0}$ corresponds to the root system $\Sigma^+_{\lambda_0,Y_0}=\Sigma^+_{\lambda_0}\cap \Sigma^+_{Y_0}$. 
We write $\pi_0(Y)=\pi_{\lambda_0,Y_0}(Y)=\prod_{\alpha\in \Sigma^+_{\lambda_0,Y_0}} \alpha(Y)$ and $c_{\lambda_0,Y_0}
=\partial(\pi_{\lambda_0,Y_0})(\pi_{\lambda_0,Y_0})$. Denote by ${\mathcal M}$ the set of positive roots that are neither in $\Sigma^+_{\lambda_0}$ nor in $\Sigma^+_{Y_0}$, \emph{i.e.} ${\mathcal M}=\Sigma^+\setminus(\Sigma^+_{\lambda_0}\cup \Sigma^+_{Y_0})$. We also write $\pi_{\mathcal M}(X)=\prod_{\alpha\in {\mathcal M} } \alpha(X)$.

\begin{prop}\label{PropPi}~
\begin{itemize}
\item[(i)] If $w\in W_Y$ then $\pi_Y(w\cdot X)=\epsilon(w)\,\pi_Y(X)$.

\item[(ii)] If $w\in W_Y$ then $ \pi_Y(\partial)[f(w\cdot Y)]=\epsilon(w)( \pi_Y(\partial)f)(w\cdot Y)$.
\end{itemize}
\end{prop}

\begin{proof}
The property (i) is well known \cite{Helgason2}. The property (ii) is straightforward
for $f(X)=e^{\langle Z, X\rangle}$ and extends by linear density.
\end{proof} 

\begin{prop}\label{MainAsympt2018}
Let $\lambda_0$, $Y_0$ be singular. The asymptotics of $\psi_{\lambda_0}(t\,Y_0 )$ when $t\to\infty$ is given by the following formula: 
\begin{equation}\label{general}
\psi_{\lambda_0}(t\,Y_0)\sim C(\lambda_0,Y_0) t^{|\Sigma^+_{Y_0}|-|\Sigma^+|}
\sum_{w\in W(\lambda_0,Y_0)} \epsilon(w) 
\pi_{Y_0}(\partial ^Y)\,\left( \pi_{\lambda_0}(w\cdot Y) e^{\langle \lambda_0, w\cdot Y\rangle}\right)\big|_{Y=t\,Y_0}
\end{equation}
where $ C(\lambda_0,Y_0) =(c_{\lambda_0} \,c_{Y_0}\,\pi'_{\lambda_0}(\lambda_0)\,\pi'_{Y_0}(Y_0))^{-1}$.

When $W(\lambda_0,Y_0)=W_{\lambda_0} \,W_{Y_0}$, the last formula simplifies to
\begin{equation}\label{withKillMax}
\psi_{\lambda_0}(t\,Y_0)\sim C_1(\lambda_0,Y_0) t^{|\Sigma^+_{Y_0}|-|\Sigma^+|}
\pi_{Y_0}(\partial ^Y)\left( \pi_{\lambda_0}(Y) e^{\langle \lambda_0, Y\rangle}\right)\big|_{Y=t\,Y_0}
\end{equation}
where $C_1(\lambda_0,Y_0)= C(\lambda_0,Y_0) |W_{\lambda_0}| \,| W_{Y_0}| / |W_{\lambda_0,Y_0}|$.
\end{prop}
\begin{proof}
We start with the alternating sum formula for the spherical function $\psi_\lambda$, written in the following way 
\begin{equation}\label{startProof}
\pi(\lambda) \pi(Y) \psi_\lambda(Y)=\sum_{w\in W} \epsilon(w) e^{\langle \lambda, w\cdot Y\rangle}
\end{equation}
We write $\pi(\lambda)=\pi_{\lambda_0}(\lambda)\pi'_{\lambda_0}(\lambda)$ and $\pi(Y)=\pi_{Y_0}(Y)\pi'_{Y_0}(Y)$.
We apply the operator
$
L=\pi_{Y_0}(\partial ^Y)
\pi_{\lambda_0}(\partial ^\lambda)
$
to both sides of \eqref{startProof}. Using the fact that $ \pi_{\lambda_0}(\partial ^\lambda) \,e^{\langle \lambda, w\cdot Y\rangle}= \pi_{\lambda_0}(w\cdot Y) e^{\langle \lambda, w\cdot Y\rangle}, $ we obtain
\begin{align*}
c_{\lambda_0} \,c_{Y_0}\,\pi'_{\lambda_0}(\lambda_0)\,\pi'_{Y_0}(t\,Y_0) \,\psi_{\lambda_0}(t\,Y_0)= \sum_{w\in W} \,\epsilon(w) 
\,\pi_{Y_0}(\partial ^Y)\,\left( \pi_{\lambda_0}(w\cdot Y) \,e^{\langle \lambda_0, w\cdot Y\rangle}\right)\big|_{Y=t\,Y_0}.
\end{align*}
In order to get the exact asymptotics of $ \psi_{\lambda_0}(t\,Y_0)$, we only need to deal with
$w\in W$ such that $\langle \lambda_0, w\cdot Y_0\rangle= \langle \lambda_0, Y_0\rangle$. This gives the asymptotics \eqref{general}.

We now assume that $W(\lambda_0,Y_0)= W_{\lambda_0} \,W_{Y_0}$. The asymptotics \eqref{general} simplify, since
by Proposition \ref{PropPi}, we obtain for $w=w_1\,w_2$ with $ w_1\in W_{\lambda_0}$ and
$w_2\in W_{Y_0}$
\begin{align*}
\pi_{Y_0}(\partial ^Y)\left( \pi_{\lambda_0}(w\cdot Y) \,e^{\langle \lambda_0, w\cdot Y\rangle}\right)&=
\epsilon(w_1) \pi_{Y_0}(\partial ^Y)\left( \pi_{\lambda_0}(w_2Y) \,e^{\langle \lambda_0, w_2Y\rangle}\right)\\
&=\epsilon(w_1)\,\epsilon(w_2)
 \pi_{Y_0}(\partial ^Y)\left( \pi_{\lambda_0}(Y) \,e^{\langle \lambda_0, Y\rangle}\right)\\
 &=\epsilon(w)
\, \pi_{Y_0}(\partial ^Y)\left( \pi_{\lambda_0}(Y) \,e^{\langle \lambda_0, Y\rangle}\right).
\end{align*}
Using Lemma \ref{strate}, we have 
$
|W_{\lambda_0} W_{Y_0}| = |W_{\lambda_0}||W_{Y_0}|/|W_{\lambda_0,Y_0}|
$. We obtain the formula \eqref{withKillMax}.
\end{proof}

\begin{theorem}\label{Final}
Let $\lambda_0$, $Y_0$ be singular.  Assume that $W(\lambda_0,Y_0)=W_{\lambda_0} \,W_{Y_0}$. Then
the asymptotics of $\psi_{\lambda_0}(t\,Y_0 )$ when $t\to\infty$ are given by the following formula: 
\begin{equation}\label{ultimate}
\psi_{\lambda_0}(t\,Y_0)\sim D(\lambda_0,Y_0)
 \,t^{-m}\, e^{t\,\langle \lambda_0, Y_0\rangle}
\end{equation}
where $m$ is the number of positive roots that are neither in $\Sigma^+_{\lambda_0}$ nor in $\Sigma^+_{Y_0}$  \emph{i.e.} 
\begin{align*}
m=card\M=|\Sigma|^+-(|\Sigma^+_{\lambda_0}|+|\Sigma^+_{Y_0}|-|\Sigma^+_{\lambda_0}\cap\Sigma^+_{Y_0}|)
\end{align*}
and
\begin{align*}
D(\lambda_0,Y_0)=\frac{c_{\lambda_0,Y_0}}{c_{\lambda_0} \,c_{Y_0}} \,
\,\frac{|W_{\lambda_0}|\, |W_{Y_0}|}{|W_{Y_0}\cap W_{\lambda_0}|} 
\frac{1} {\pi_\M(\lambda_0)\,\pi_{\M}(Y_0)}.
\end{align*}
\end{theorem}

\begin{remark}
When $Y_0$ is regular, the method of proof used in Theorem \ref{Poiss} for the asymptotics of the Poisson kernel could have been used here. When both $\lambda_0$ and $Y_0$ are singular, that approach fails to apply.
\end{remark}

\begin{proof}
Using Leibniz formula, we have
\begin{align*}
\lefteqn{\pi_{Y_0}(\partial ^Y)\,\left( \pi_{\lambda_0}(Y) \,e^{\langle \lambda_0, Y\rangle}\right)\big|_{Y=t\,Y_0}}\\
&=\pi_0(\partial^Y)\,\prod_{\alpha\in \Sigma^+_{Y_0}\setminus\Sigma^+_{\lambda_0}}\,\partial^Y(A_\alpha)
\,\left(
\pi_{\lambda_0}(Y) e^{\langle \lambda_0,Y\rangle}
\right)\big|_{Y=t\,Y_0}\\
&=\pi_0(\partial^Y)
\,\left(
\prod_{\alpha\in \Sigma^+_{Y_0}\setminus\Sigma^+_{\lambda_0}}\,\langle\lambda_0,\alpha\rangle\,
\pi_{\lambda_0}(Y)
\,e^{\langle \lambda_0, Y\rangle}
+P(Y)\,e^{\langle \lambda_0, Y\rangle}\right)\big|_{Y=t\,Y_0}\\
\end{align*}

The number of factors in each term $P(Y)$ of the form $\langle \eta,Y\rangle$ where $\eta$ is a root, is strictly less than the number of factors in $\pi_{\lambda_0}$	\emph{i.e.} less than $|\Sigma^+_{\lambda_0}|$.
	
In the expression in the last line, all derivatives involving the term $e^{\langle \lambda_0, Y\rangle}$ give 0 since $\beta(\lambda_0)=0$ 	for $\beta\in \Sigma^+_{Y_0}\cap\Sigma^+_{\lambda_0}$.

In the derivatives of $\pi_{\lambda_0}(Y)$, any term that contains $\langle \beta,Y\rangle$ with $\beta\in \Sigma^+_{Y_0}\cap\Sigma^+_{\lambda_0}$ will be zero when $Y$ is replaced by $t\,Y_0$. Thus, for a nonzero result, the operator 
$\pi_{0}(\partial^Y)$ 	must be applied to $\pi_{0}(Y)$, what gives 	$c_{\lambda_0,Y_0}>0$. We obtain
\begin{align*}
\lefteqn{\pi_{Y_0}(\partial ^Y)\,\left( \pi_{\lambda_0}(Y) \,e^{\langle \lambda_0, Y\rangle}\right)\big|_{Y=t\,Y_0}}\\
&=\prod_{\alpha\in \Sigma^+_{Y_0}\setminus\Sigma^+_{\lambda_0}}\,\langle\lambda_0,\alpha\rangle\,
\prod_{\gamma\in \Sigma^+_{\lambda_0}\setminus\Sigma^+_{Y_0}}\,\langle \gamma,t\,Y_0\rangle\,
c_{\lambda_0,Y_0}
\,e^{t\,\langle \lambda_0, Y_0\rangle}
+\pi_{0}(\partial^Y)P(t\,Y_0)e^{t\langle \lambda_0, Y_0\rangle}\\
&=c_{\lambda_0,Y_0}\,t^{|\Sigma^+_{\lambda_0}|-|\Sigma^+_{\lambda_0}\cap\Sigma^+_{Y_0}|}\,\prod_{\alpha\in \Sigma^+_{Y_0}\setminus\Sigma^+_{\lambda_0}}\,\langle\lambda_0,\alpha\rangle\,
\prod_{\gamma\in \Sigma^+_{\lambda_0}\setminus\Sigma^+_{Y_0}}\,\langle \gamma,Y_0\rangle\,
\,e^{t\,\langle \lambda_0, Y_0\rangle}
+\hbox{negligible terms}.
\end{align*}

We labeled as ``negligible terms'' the terms with the derivatives involving $P(Y)$. 
They	have the number of factors of the form $\langle \eta,t\,Y_0\rangle$ strictly less then $|\Sigma^+_{\lambda_0}|-|\Sigma^+_{\lambda_0,Y_0}|$, so 
 strictly less than the term $\prod_{\gamma\in \Sigma^+_{\lambda_0}\setminus\Sigma^+_{Y_0}}\,\langle \gamma,t\,Y_0\rangle$. 
The rest follows from the definition of $C(\lambda_0,Y_0)$.
\end{proof}

\begin{remark}
We can give a more explicit expression for the constant $D$, using the formula
\begin{align*}
\partial(\pi)\,\pi= \frac{|W|\,\pi(\rho)}{2^\gamma},
\end{align*}	
where $\rho=\frac{1}{2} \,\sum_{\alpha\in\Sigma^+}\,m_\alpha \alpha=	\sum_{\alpha\in\Sigma^+} \,\alpha$ and $\gamma=|\Sigma^+|$ is the number of positive roots.

For $X\in \a$, denote $p_X=\pi_X(\rho_X)$. Analogously, we define $p_{X_1,X_2}$ for the root system annihilating both
elements $X_1,X_2\in \a$.  We have 
\begin{align*}
D(\lambda_0,Y_0)=
\frac{2^{\gamma_{\lambda_0,Y_0}-\gamma_{\lambda_0}-\gamma_{Y_0}}}
{\pi_\M(\lambda_0)\,\pi_{\M}(Y_0)}\,\frac{p_{\lambda_0,Y_0}}{p_{\lambda_0} \,p_{Y_0}},
\end{align*}
and therefore
\begin{align*}
\lim_{t\to \infty} \frac{\psi_{\lambda_0}(t\,Y_0)}{t^{-m}\,e^{t\langle \lambda_0,Y_0\rangle}}=
\frac{2^{\gamma_{\lambda_0,Y_0}-\gamma_{\lambda_0}-\gamma_{Y_0}}}
{\pi_\M(\lambda_0)\,\pi_{\M}(Y_0)}\frac{p_{\lambda_0,Y_0}}{p_{\lambda_0} \,p_{Y_0}}
\end{align*}
\end{remark}

\begin{remark}
As a quick  application of Theorem \ref{Final},  we find, in the flat complex case,  a simple proof of a general result
 of Vogel and Voit:  for  symmetric spaces with subexponential (here polynomial) growth, the set of bounded spherical functions coincides with the support of
the Plancherel measure of the associated Gelfand pair $(G_0 ,K)$, $G_0$ the Cartan motion
group. See for instance Sections 3.2 and 3.3. of \cite{RV}  for details.
A proof in the flat complex case was also proposed by Helgason in \cite{Helgason1}.
\end{remark}

\begin{remark}\label{pub}
Taking into account the relationship between the spherical functions in the flat case and those in the curved case for the complex Lie groups, the estimates of spherical functions in \cite{Narayanan, Schapira} extend to  the flat case.

In this case, Theorem \ref{Final} completes the estimates of \cite{Narayanan, Schapira} providing the exact asymptotics.
We conjecture that asymptotics with appropriate constants and not only estimates hold in the results of Narayanan, Pasquale and Pusti \cite{Narayanan} and Schapira \cite{Schapira}. 

The asymptotic expansion given in \cite[Proposition 3.8]{Barlet} for regular $\lambda_0$ and $Y_0$
implies asymptotics of spherical functions in this case.
Theorem \ref{Final} strengthens this result to singular $\lambda_0$ and $Y_0$.

\end{remark}

\begin{prop}\label{K}
Let $X$ and $Y$ be singular and $m'=|\Sigma_X^+\cup\Sigma_Y^+|$.
With the same notation as in Theorem \ref{Final}, we have
\begin{align*}
p_t^W(X,Y)&\sim\frac{D(X,Y)\,2^{m-d} }{|W|\,\pi^{d/2}\,\pi(\rho)} \,t^{-\frac{d}{2}-m'} \,e^{\frac{-|X-Y|^2}{4t}}
\end{align*}
as $t\to0^+$.
\end{prop}

\begin{proof}
From Theorem \ref{Final}, we have for $t>0$ close to 0,
\begin{align*}
\psi_{X}(Y/(2\,t))\sim D(X,Y)
 \,(2\,t)^{m}\, e^{\langle X, Y/(2\,t)\rangle}.
\end{align*}
Combined with \eqref{heat-spherical}, this leads us to
\begin{align*}
p_t^W(X,Y)&\sim\frac{2^m}{|W|\,2^d\,\pi^{d/2}\,\pi(\rho)} \,t^{-\frac{d}{2}-\gamma} \,e^{\frac{-|X|^2-|Y|^2}{4t}}\,
D(X,Y) \,t^{m}\, e^{\langle X, Y/(2\,t)\rangle}\\
&=\frac{D(X,Y) \,2^m}{|W|\,2^d\,\pi^{d/2}\,\pi(\rho)} \,t^{-\frac{d}{2}-(\gamma-m)} \,e^{\frac{-|X-Y|^2}{4t}}.
\end{align*}
\end{proof}

%%%%%%%%%%%%%%%%%%%%%%%%%%%%%%%%%%%%%%%%%%%%%%%%%%%%%%%%%%%%%%%%%%%%%%%%%%%%%%%%%%%%%%%%%%%%%%%%%%%%%%%%%%%
%%%%%%%%%%%%% DYSON  %%%%%%%%%%%%%%%%%%%%%%%%%%%%%%%%%%%%%%%%%%%%%%%%%%%%%%%%%%%%%%%%%%%%
\section{Applications to the Dyson Brownian motion and  stochastic analysis}\label{DYSON}

\subsection{Definition and transition density of the Dyson Brownian motion}

When a probabilist looks at formula \eqref{operator}, he or she 
sees in it the generator of the Doob $h$-transform (refer to \cite{ry}) of the Brownian Motion on $\R^d$ with the excessive function $h(X)=\pi(X)$. For the root system $A_d$ on $\R^d$, the operator $\Delta^W$
restricted to  functions on $\a^+$,
is the generator of the Dyson Brownian Motion on $\a^+\subset \R^d$ (\cite{Dyson}), {\it i.e.} the $d$ Brownian independent particles $B^{(1)}_t,\ldots,B^{(d)}_t $ conditioned not to collide.
More generally, for any root system $\Sigma$ on $\R^d$, the construction of a Dyson Brownian Motion as a Brownian Motion conditioned not to touch the walls of the positive Weyl chamber, can be done ith a starting point $X\in\overline{\a^+}$ (\cite{Grabiner}).

Let us recall basic facts about the Doob $h$-transform and the Dyson Brownian Motion.
 Let $\Sigma$ be a root system on $\R^d$ and $\pi(X)=\prod_{\alpha>0}
\,\langle \alpha,X\rangle$. It is known that $\pi$ is $\Delta_{\R^d}$-harmonic on $\R^d$ (\cite{Grabiner}), so in particular $\pi$ is excessive.
%%%%%%%%%%%%%%%%%%%%%%%%%%%%%%%%%%%%%%%%%%%
\begin{definition}
Let $\Sigma$ be a root system on $\R^d$ and $\pi(X)=\prod_{\alpha>0}
\,\langle \alpha,X\rangle$.
The Dyson Brownian Motion $D_t^\Sigma$ on the positive Weyl chamber $\a^+$ is defined as the $h$-Doob transform of the Brownian Motion on $\R^d$, with $h=\pi$, {\it i.e.}
its transition density is equal to
\begin{equation}\label{heatDyson}
p^{\rm D}_t(X,Y)=\frac{\pi(Y)}{\pi(X)} p^{\rm killed}_t(X,Y), \qquad X\in\overline{\a^+}
,Y\in\a^+,
\end{equation}
where $p^{{\rm killed}}_t(X,Y)$ is the transition density of the Brownian Motion killed at the first strictly positive time of touching $\partial \a^+$. 
\end{definition}
 The infinitesimal generator of $D_t^\Sigma$ is given by the formula (\cite{ry}) 
\begin{equation} \label{DysonGen}
\Delta^D f= \pi^{-1} \Delta^{\R^d} (\pi f),\qquad \text{supp}\, f\subset \a^+,
\end{equation}
which coincides on $\a^+$ with formula \eqref{operator} for $\Delta^W$.
The only differences with the symmetric flat complex case are that 
the domain of kernels ${\mathcal K}^D(X,Y)$
is restrained to  $X\in\overline{\a^+}
,Y\in\a^+$, and  that
no invariant measure $\pi^2(Y)\,dY$ appears for the integral kernels in the Dyson Brownian Motion case. Consequently, we obtain

 %%%%%%%%%%%%%%%%%%%%%%%%%%%%%%%%%%%%%%%%%%%%
\begin{cor}\label{DysonHeat}
The transition density and the heat kernel of the Dyson Brownian Motion $D_t^\Sigma$ on $\a^+\subset\R^d$ is given by the formula
\begin{align*}\label{PDyson}
p^{\rm D}_t(X,Y)=\frac{\pi(Y)}{
\,\pi(X)} \,\sum_{w\in W} \epsilon(w) \,h_t(X-w\cdot Y), 
\qquad X\in\overline{\a^+}
,Y\in\a^+, \end{align*}
where $h_t(X-Y)=\frac{1}{(4\,\pi\,t)^{d/2}}e^{-\frac{|X-Y|^2}{4\,t}}$ is the Euclidean heat kernel on $\R^d$.

In the case $\Sigma=A_p$ we have 
\begin{align*}
p^{\rm D}_t(X,Y)=\frac{\pi(Y)}{\,\pi(X)} \,\det\left(g_t(x_i,y_j)\right),
\qquad X\in\overline{\a^+}
,Y\in\a^+, 
\end{align*}
where $g_t(u,v)=\frac{1}{\sqrt{4\,\pi t}}\,e^{-|u-v|^2/4t}$ is the 1-dimensional classical heat kernel.
\end{cor}
\begin{proof}
We use Theorem \ref{kernels}~(1) and Corollary \ref{DetAd}.
\end{proof}

Comparing the formulas from Corollary \ref{DysonHeat} with formula \eqref{heatDyson}, we obtain the following
formulas for the heat kernel of the Brownian Motion killed at the first strictly positive time of touching 	a wall of the positive Weyl chamber. 
\begin{cor}\label{KilledHeat}
The transition density for the Brownian Motion killed when exiting the positive Weyl chamber is given by the formula
\begin{equation}\label{killedHeat}
p^{\rm killed}_t(X,Y)= \,\sum_{w\in W}\, \epsilon(w) \,h_t(X-w\cdot Y). 
\end{equation}
In the case $\Sigma=A_{d-1}$ we have 
\begin{align}\label{KarlinMACG}
p^{\rm killed}_t(X,Y)= \,\det\left(g_t(x_i,y_j)\right),
\end{align}
\end{cor}
\begin{remark}
Karlin and McGregor \cite{Karlin} showed formula \eqref{KarlinMACG} by different methods. 
In \cite{Grabiner}, formulas for $p^{\hbox{killed}}_t(X,Y)$ for the root systems $B_d$, $C_d$ and $D_d$ are proven. Our method of alternating sums provides a simple proof of
formula \eqref{killedHeat} valid for any root system $\Sigma$.
\end{remark}

%%%%%%%%%%%%%%%%%%%%%%%%%%%%%%%%%%%%%%%%%%

\subsection{Poisson and Newton kernels for the Dyson Brownian Motion} \label{DBM}

The Poisson and Newton kernels
$P^D(X,Y)$ and $N^D(X,Y)$ are central objects
of the potential theory  of the Dyson Brownian Motion $D_t^\Sigma$ and this is a first reason of studying them.
However,
these kernels have stochastic  interpretation and, consequently,
are  useful in  stochastic analysis of the Dyson Brownian Motion.

Denote by $D_t^{\Sigma,X}$ the  
Dyson Brownian Motion starting from
$X$. 
Let $X\in B(0,1)$ and $$T(X)=\inf\{t>0\,|\ D_t^{\Sigma,X}\not\in B(0,1)\}.$$
By the mean-value theorem  for harmonic functions of general strong Markov processes
(see \cite {Dynkin, Kall}),  called sometimes
Kakutani's Theorem
(\cite{Chung}),
the Poisson 
kernel $ P^D(X,Y)$ is the density of  the
random vector
$$D_{T(X)}^{\Sigma,X}$$
on the sphere. This is the Dyson Brownian Motion starting from
$X$ inside the unit ball and 
stopped at the first time $T(X)$ of exiting the 
ball.  If $dY$ denotes the Lebesgue measure on
the unit sphere, then  $ P^D(X,Y)dY$ is called the harmonic measure of the  
Dyson Brownian Motion on the unit sphere.

The Newton kernel $N^D(X,Y)$
is related to the transition probability $p_t^D(X,Y)$ of the  
Dyson Brownian Motion
by the formula (\cite
%[Chapter 2.3]
{Chung})
$$N^D(X,Y)=\int_0^\infty p_t^D(X,Y)  dt.  $$

The alternating sum formulas for the  integral Poisson and Newton kernels $P^D$ and $N^D$ 
of the Dyson Brownian Motion $D_t^\Sigma$ can be easily deduced from their counterparts (see Theorem \ref{kernels}) for the flat complex symmetric spaces $M$, just by multiplying $P^W$ and $N^W$
by $\pi(Y)^2$.

\begin{remark}\label{PNp}
We have  
\begin{align*}
P^D(X,Y)&=|W|\,\pi(Y)^2 \,P^W(X,Y),\\
N^D(X,Y)&=|W|\,\pi(Y)^2 \,N^W(X,Y)~\hbox{and}\\
p_t^D(X,Y)&=|W|\,\pi(Y)^2 \,p_t^W(X,Y).
\end{align*}
The Poisson kernel of the Dyson Brownian Motion extends continuously to
$X\in\overline{\a^+}$, $Y\in \overline{\a^+}$.
In particular, $P^D(X,Y)=0$     when $Y$ is singular.
The same remarks apply to the Newton kernel $N^D(X,Y)$ and to the heat kernel.

These observations allow us to consider the ratios $P^D(X,Y)/\pi(Y)^2$,
$N^D(X,Y)/\pi(Y)^2$ and $p_t^D(X,Y)/\pi(Y)^2$ even when $Y\in\partial\a^+$.
\end{remark}

Theorems \ref{Poiss} and \ref{Newt} imply asymptotics for the Poisson and Newton kernels for the Dyson Brownian Motion. 
For completeness and for their applications in the potential theory and in the stochastic analysis of the process $D_t^\Sigma$, we state these results here.
 
\begin{cor}\label{Dyson}
The following formulas hold for
 $X,Y\in\overline{\a^+}$
\begin{align*}
P^{\rm D}(X,Y)&=\frac{(1-|X|^2)\,\pi(Y)}{\,w_d\,\pi(X)} \,\sum_{w\in W}\,\frac{\epsilon(w)}{|X-w\cdot Y|^{d}}\\
N^{\rm D}(X,Y)&=\frac{\pi(Y)}{2\,\pi\,\pi(X)}\,\sum_{w\in W}\,\epsilon(w)\,\ln|X-w\cdot Y|~~ \hbox{when $d=2$},\\
N^{\rm D}(X,Y)&=\frac{\pi(Y)}{\,(2-d)\,w_d\,\pi(X)}
\,\sum_{w\in W}\,\frac{\epsilon(w)}{|X-w\cdot Y|^{d-2}}~~
\hbox{when $d\geq3$}.\nonumber
\end{align*}
\end{cor}

Keeping in mind Remark \ref{PNp}, equations \eqref{asymptPoisson} and \eqref{asymptNewton} lead us to the following result.

\begin{cor}\label{Poiss_Dyson}
Let $Y_0\in \overline{\a^+}$, $\Sigma'=\{\alpha\in\Sigma|\ \alpha(Y_0)=0\}$, $\Sigma'_+=\Sigma'\cap\Sigma^+$, $\gamma'=|\Sigma_+'|$ and $\pi'(X)=\prod_{\alpha\in \Sigma_+'}\,\langle\alpha,X\rangle$.
\begin{itemize}
\item[(i)] Let $Y_0\in \partial B$.
Then 
\begin{equation*}
\frac{P^{\rm D}(X,Y_0)}{\pi'(Y_0)^2}\stackrel{Y_0}{\sim} \frac{2^{2\,\gamma'}\,(d/2)_{\gamma'}}{\,w_d\,\pi'(\rho')}\,\frac{1-|X|^2}{|X-Y_0|^{2\gamma'+d}}.
\end{equation*}

\item[(ii)] If $d=2$, $\alpha$, $\beta$ are the simple roots, $\alpha(Y_0)\not=0$, $\beta(Y_0)=0$, then
\begin{align*}
\frac{N^{\rm D}(X,Y_0)}{\pi'(Y_0)^2}\stackrel{Y_0}{\sim} 
\frac{-2^{2\,\gamma'-1}\,(\gamma'-1)!}{2\,\pi
\,{\langle\alpha,\alpha\rangle}}
\,|X-Y_0|^{-2}.
\end{align*}

\item[(iii)] If $d\geq 3$, then
\begin{align*}
\frac{N^{\rm D}(X,Y_0)}{\pi'(Y_0)^2}\stackrel{Y_0}{\sim} \frac{2^{2\,\gamma'}
\,((d-2)/2)_{\gamma'}}{(2-d)\,w_d\,\pi(\rho')}\,\frac{1}{|X-Y_0|^{2\gamma'+d-2}}.
 \end{align*}
\end{itemize}
\end{cor}

\subsection{On the transition probability of the Dyson Brownian Motion} 

The heat kernel $p^D_t(X,Y)$
of the Dyson Brownian Motion is nonzero  for $X\in\overline{\a^+}$
and $Y\in\a^+$ and zero if $Y\in\partial\a^+$ as per Remark \ref{PNp}. 
By Proposition \ref{K}
we then have the following asymptotic result:

\begin{cor}\label{KDyson}
Let $X$ and $Y$ be singular and $m'=|\Sigma_X^+\cup\Sigma_Y^+|$.
With the same notation as in Theorem \ref{Final}, we have
\begin{align*}
\frac{p_t^D(X,Y)}{\pi(Y)^2}&\sim\frac{D(X,Y)\,2^{m-d} }{|W|\,\pi^{d/2}\,\pi(\rho)} \,t^{-\frac{d}{2}-m'} \,e^{\frac{-|X-Y|^2}{4t}}
\end{align*}
as $t\to0^+$.
\end{cor}

\begin{remark}
In \cite{KT, Koenig}  an asymptotic formula in terms of Schur functions was used to analyze the  heat kernel of Dyson Brownian Motion.
\end{remark}

We are grateful to one of the anonymous referees for suggesting that  an  asymptotic result for Dyson heat kernel
was a natural extension of our results. It lead us to Corollary  \ref{KDyson}.

\subsection{Remarks on relations to stochastic analysis } \label{relations}

At the beginning of  Section \ref{DBM} and in the asymptotic formulas for the Poisson and Newton kernels in
Corollary \ref{Dyson}, the factor $\pi(Y)^2$  appears as a common feature. In the context of random matrix theory and non-colliding
diffusive particle problems (the original Brownian motion models), this factor is very important as follows.

\begin{itemize}
\item[(i)] This factor $\pi(Y)^2$
appearing in the probability density becomes zero if $x_j = x _i$. Then the system has some ``repulsive'' interaction and it will be
regarded as a determinantal (Fermion) point process.

\item[(ii)] The  squared Vandermonde determinant $\pi(Y)^2$ can be written as the determinant of a matrix whose entries are given by orthonormal polynomials.
This opens the way  to applications
to reproducing kernels of
Hilbert spaces spanned by these orthonormal functions.

\item[(iii)] The  factor $\pi(Y)^2$  is a special case with $\beta = 2$ in the general setting $\prod^N_{j=1}\, (x_j-x_i)^\beta$  
important in the theory of random matrices.
\end{itemize}

It is natural to ask whether it is possible to discuss the
O'Connell and Macdonald
stochastic processes studied in \cite{BC14,K12,O12} from the viewpoint of the present paper. The multivariate processes studied there are related to the representation theory (e.g.{} Gelfand-Zetlin patterns), the symmetric functions and special functions (e.g.{} Whittaker functions, Macdonald polynomials), and integrable systems (e.g.{} quantum Toda lattice).  This question is best left to another paper.

\subsection{ Curved case and relations to Schr\"odinger operators}

The alternating sum formulas given in Section \ref{flatFormulas} have analogs in the curved complex case, considered in this section.
To underline the difference with the flat case, we denote the spherical and potential analysis objects on $M$ with a tilde ($\tilde{\phantom{.}}$).
The kernels in this section are with respect to the invariant measure $\delta(Y)\,dY$ where
\begin{align*}
\delta(Y)=\prod_{\alpha>0}\,\sinh^2\alpha(Y).
\end{align*}

The following method of construction of kernels is similar to the one presented in Section \ref{flatFormulas}.

\begin{enumerate}
\item Exploit the formula for the Laplace-Beltrami operator on $M$ (\cite[Chap.{} II, Theorem 5.37]{Helgason3}):
\begin{align*}
\tilde\Delta^W f= \delta^{-1/2}\, (\Delta^{\R^d}-|\rho|^2)(\delta^{1/2} \,f)
\end{align*}
	
\item Apply the $W$-invariance.
\end{enumerate}	
In this way, the Euclidean kernel ${\mathcal K}^{\Delta^{\R^d}-|\rho|^2}(X,Y)$ (heat, potential, Poisson, \dots) for the operator
$\Delta^{\R^d}-|\rho|^2$ is transformed into the kernels $\tilde{\mathcal K}$ for $G/K$:
\begin{equation}\label{WMethodcurve}
\tilde{\mathcal K}(X,Y)=\frac{1}{\delta^{1/2}(X)\,\delta^{1/2}(Y)}
\,\sum_{w\in W}\,{\epsilon(w)} {\mathcal K}^{\Delta^{\R^d}-|\rho|^2}(X,w\cdot Y).
\end{equation}

The estimates of the Newton kernel
$\tilde N(X,Y)$ for all curved Riemannian symmetric spaces $G/K$ were obtained in  \cite{Anker1}.
In the  case when $G$ is complex,   it would be possible to apply our methods based on formula  \eqref{WMethodcurve}, using the 
 knowledge of the Newton kernel  $N(X,Y)$  of the Schr\"odinger operator $\Delta^{\R^d}-|\rho|^2$,
\emph{i.e.} the $|\rho|^2$-potential ($|\rho|^2$-resolvent) of $\Delta^{\R^d}$.
The Newton kernel  $N(X,Y)$ may be expressed with the Bessel function of third type $K_{d/2}$.

For the Poisson kernel for $\Delta^W$, we need to know the Poisson kernel for the Schr\"odinger operator $\Delta^{\R^d}-|\rho|^2$. This kernel is not known explicitly.
However much intensive work was  and is presently being done
in the analytic and stochastic theory of heat and other kernels for Schr\"odinger operators, see e.g. \cite{Bogdan, Bogdan2, Chung}. 
In a further work, we plan to study
thoroughly these results and apply them to the estimates of the Poisson kernel on curved complex symmetric spaces.

%%%%%%%%%%%%%%%%%%%%%%%%%%%%%%%%%%%%%%%%%%%%%%%%%

\appendix

\section{The Killing-max property}\label{KM}

The aim of this appendix is to find precise conditions on $w\in W$ under which 
\begin{equation}
\label{main}
\langle \lambda, w\cdot Y \rangle=\langle \lambda, Y \rangle.
\end{equation}

\begin{definition}\label{WLwY}
Let $W_\lambda=\{w\in W\colon w\cdot \lambda=\lambda\}$ (similarly for $W_Y$). We will say that the property Killing-max is satisfied if
\eqref{main} is verified if and only if $w\in W_\lambda\,W_Y$.
\end{definition}

\begin{remark}\label{know}
It is clear that the condition $w\in W_\lambda\,W_Y$ is sufficient. Property Killing-max is also satisfied whenever at least one of $\lambda$ or $Y$ is regular (refer to \cite{Helgason1}). We observe also that this property only depends on the action of the Weyl group on the Cartan subalgebra $\a$. Given that $\langle \lambda, w\cdot Y \rangle=\langle w^{-1} \lambda, Y \rangle$, this problem is symmetric in $\lambda$ and $Y$.
\end{remark}

In Table \ref{Weylaction}, we describe the action of the Weyl group on the Cartan subalgebra in the case of the noncompact and complex simple Lie algebras.
Note that in the case of $(\f_{4(-26)},\so(9))$, which is not in the table, the Killing-max property is trivially true since the rank of the space is 1.

\begin{table}[ht]
\begin{tabular}{|p{5.5cm}|p{4.0cm}|p{3.0cm}|p{2.0cm}|}\hline
Symmetric space&Description of $X\in\a^+$&Action of $w\in W$, the Weyl group &Underlying root system\\ \hline
$\SL(n,\F)/\SU(n,\F)$,\break$\F=\R$, $\C$, $\H$, $n\geq 2$,\hfill\break
$\F={\bf O}$, $n=3$ (\emph{i.e.} ${\bf E}_6/{\bf F}_4$)
&$X=\diag[x_1,\dots,x_n]$,\hfill\break $\sum_{i=1}^n\,x_i=0$,\hfill\break $x_1>\dots>x_n$
&$w\in S_n$ permutes the entries $x_i$&$A_{n-1}$\\\hline
$\SO(p,q)/\SO(p)\times\SO(q)$, $1\leq p<q$,\hfill\break
$\SU(p,q)/\SU(p)\times\SU(q)$ and\hfill\break
$\Sp(p,q)/\Sp(p)\times\Sp(q)$, $1\leq p\leq q$,
&$X=\left[\begin{array}{ccc} 0&\mathcal{D}_X&0\\
\mathcal{D}_X&0&0\\
0&0&0
\end{array}\right]$,\hfill\break $\mathcal{D}_X=\diag[x_1,\dots,x_p]$,
\hfill\break $x_1>\dots>x_p>0$
&$w$ permutes the $x_i$'s and changes any number of signs&$B_{n}$\\ \hline
~\break~\break$\SO(p,p)/\SO(p)\times\SO(p)$, $p\geq 2$
&$X=\left[\begin{array}{ccc} {0}& {\mathcal{D}_X}\\
 {\mathcal{D}_X}&{0}
\end{array}\right]$,\hfill\break $\mathcal{D}_X=\diag[x_1,\dots,x_p]$,\hfill\break $x_1>\dots>x_{p-1}>|x_p|$
&$w$ permutes the $x_i$'s and changes any {\bf even} number of signs&$D_n$\\ \hline
~\break$\SO^*(2\,n)/\U(n)$, $n\geq 3$&$X=\left[
\begin{array}{c|c}
0_{n\times n}&\mathcal{E}_X\\ \hline
-\mathcal{E}_X&0_{n\times n}
\end{array}
\right]$,\hfill\break $\mathcal{E}_X=\sum_{k=1}^{[n/2]}\,x_k
\,F_{2\,k,2\,k+1}$,\hfill\break
$x_1>\dots>x_{n/2}>0$&$w$ permutes the $x_i$'s and changes any number of signs&$B_{n}$\\ \hline
$\Sp(n,\R)/\U(n)$ and\hfill\break$\Sp(n,\C)/\Sp(n)$, $n\geq1$&$\left[\begin{array}{ccc} {0}& {i\,\mathcal{D}_X}\\
 {-i\,\mathcal{D}_X}&{0}
\end{array}\right]$,\break $\mathcal{D}_X=\diag[x_1,\dots,x_p]$,\hfill\break $x_1>\dots>x_{p-1}>x_p>0$&$w$ permutes the $x_i$'s and changes any number of signs&$C_{n}$\\ \hline
$\SO(2\,n,\C)/\SO(2\,n)$, $n\geq 3$&$ X=i\,\sum_{k=1}^n\,x_k\,F_{2\,k-1,2\,k}$,\break
$x_1>\dots>x_{p-1}>|x_p|$&$w$ permutes the $x_i$'s and changes any {\bf even} number of signs&$D_n$\\ \hline
$\SO(2\,n+1,\C)/\SO(2\,n+1)$, $n\geq 2$&$ X=i\,\sum_{k=1}^n\,x_k\,F_{2\,k-1,2\,k}$,\break $x_1>\dots>x_{p-1}>x_p>0$&$w$ permutes the $x_i$'s and changes any number of signs&$B_{n}$\\ \hline
${\bf F}_4^\C/{\bf F}_4$, $(\f_{4(4)},\sp(3)+\su(2))$
&$ X=[x_1,x_2,x_3,x_4]$,\hfill\break $x_2>x_3>x_4>0$, $x_1>x_2+x_3+x_4$&Refer to \cite{Cahn}&$F_4$\\ \hline
${\bf G}_2^\C/{\bf G}_2$, $(\g_{2(2)},\su(3)+\su(2))$ &$ X=\diag[x_1,x_2,x_1-x_2,0,x_2-x_1,-x_2,-x_1]$, $x_1>x_2>x_1/2$ &Refer to \cite{Moy}&$G_2$\\ \hline
\end{tabular}
\caption{Action of the Weyl group (except for $E_6$, $E_7$ and $E_8$) \label{Weylaction}}
\end{table}

\subsection{Type $A_n$ ($\sl(n+1,\F)$)}

\begin{lemma}[``max principle'' for permutations]\label{principle}
Let $\lambda$, $Y\in\R^n$ with their entries in decreasing order and let $w\in S_n$ be a permutation.	Suppose that the block of $\lambda_1$ in $\lambda$ has length $j_0\geq 1$ and that the block of $Y_1$ in $Y$ has length $i_0\ge 1$. If $\min w^{-1}(\{1,\ldots, i_0\} )>j_0$ then $\langle \lambda, w\cdot Y \rangle\ < \ \langle \lambda, Y \rangle$.
\end{lemma}

\begin{remark}
The lemma states that if $\langle \lambda, w\cdot Y \rangle=\langle \lambda, Y \rangle$ then the permutation $w$ is such that ``$\max Y$ meets $\max \lambda$'', \emph{i.e.} there exists $i\le j_0$ such that $(w\cdot Y)_i=y_1$.
\end{remark}

\begin{proof}
Without loss of generality, we may assume that $\lambda\not=\lambda_1\,1^n$ and $Y\not=y_1\,1^n$.
Let $i=\min w^{-1}(\{1,\ldots, i_0\})$. By assumption, the first $y_1$ appears in $w\cdot Y$ at the $i$-th position with $i>j_0$. 
Let $w(1)=k$, \emph{i.e.} $w\cdot Y$ begins with $y_k$. We have $y_k<y_1$ and $\lambda_i<\lambda_1$.
Consider $w_0=(1i~)\,w$; we then have
\begin{align*}
\langle \lambda, w_0\cdot Y \rangle-\langle \lambda, w\cdot Y \rangle\ =(\lambda_1-\lambda_i)\,(y_1-y_k)>0.
\end{align*}

By the standard property of the Weyl group, $\langle \lambda, w_0\cdot Y \rangle \leq \langle \lambda, Y \rangle$. Hence, $\langle \lambda, w\cdot Y \rangle < \langle \lambda, Y \rangle$.
\end{proof}

\begin{cor}\label{An}
Property Killing-max is verified in the case of the root system $A_n$.
\end{cor}

\begin{proof}
We use the same notation as in Lemma \ref{principle} and in its proof. Suppose $\langle \lambda, Y \rangle=\langle \lambda, w\cdot Y \rangle$. 
We use induction on $n$. The result is clear for $n=1$.
By Lemma \ref{principle}, there exists $i\leq j_0$ such that $w(i)\leq i_0$.

We now apply the induction hypothesis to $\lambda=(\lambda_2,\dots,\lambda_n)$ and to $Y=(y_2,\dots,y_n)$.
Let $\lambda'=(\lambda_2,\dots,\lambda_n)$, $Y'=(y_2,\dots,y_n)$ and note that $[y_{w(1)},\dots,\widehat{y_{w(i)}},\dots,y_{w(n)}]$ is a permutation of $Y'$ (say $w'\cdot Y'$). We have $\langle \lambda, w\cdot Y \rangle=\lambda_i\,y_1+\sum_{k=2}^n\,\lambda_k\,y_{w'(k)}$ where $w'$ is a permutation of $\{2,\dots,n\}$. We then have $\langle \lambda', Y' \rangle=\langle \lambda, w'\cdot Y' \rangle$ since $\lambda_i=\lambda_1$. By the induction hypothesis $w'=w_{\lambda'}\,w_{Y'}\in W_{\lambda'}\,W_{Y'}$.
We extend $w_{\lambda'}$ and $w_{Y'}$ to $w_{\lambda}\in W_\lambda $ and $w_{Y}\in W_Y$ by having them fix $1$ in both cases. With the permutation $w_0=(1~i)\in W_\lambda$, we have $w=(1~i)\,w_\lambda\,w_Y\in W_\lambda\,W_Y$. 
\end{proof}

\subsection{Type $B_n$ ($\so(2\,n+1,\C)$) and $C_n$ ($\sp(n,\C)$)}

\begin{prop}\label{KillBC}
Property Killing-max is verified in the case of the root systems $B_n$ and $C_n$.
\end{prop}

\begin{proof}
Recall that $B_n$ is the root system of $\so(2n+1,\C)$. 	The positive Weyl chamber is defined by the condition 
\begin{align*}
\lambda_1>\lambda_2>\dots > \lambda_n> 0
\end{align*}

The Weyl group is $W=S_n\rtimes \{\pm1\}^n$; its elements are called ``signed permutations''. It is straightforward to see that  sign changes in $w\cdot Y$ strictly decrease $\langle \lambda, Y \rangle$ unless the negative terms in $w\cdot Y$ are in front of $\lambda_i=0$.

More precisely, if $w\cdot Y$ has strictly negative terms in positions where $\lambda_i>0$, then $ \langle \lambda, w\cdot Y \rangle < \langle \lambda, w_0\,w\cdot Y \rangle \le \langle \lambda, Y \rangle$ where $w_0$ changes the negative signs in $w\cdot Y$ into positive ones.

Thus, if \eqref{main} holds, all negative terms in $w\cdot Y$ are in front of $\lambda_i=0$.
Then $w_0\in W_\lambda$ and $ \langle \lambda, w\cdot Y \rangle = \langle \lambda, w_0\,w\cdot Y \rangle$.
All the terms of $ w_0w\cdot Y$ are non-negative and the result for $A_n$ applies.

To conclude, it suffices to recall that $C_n$ is the root system for $\sp(n,\C)$. We have $W(C_n)=W(B_n)$, the only difference is in the relative length of roots (\cite[p.{} 227]{Erd}).
\end{proof}

\subsection{Type $D_n$ ($\so(2\,n,\C)$)}

The Weyl group $W$ is composed by permutations and the signs change {\bf by pairs}, \emph{i.e.} of two terms simultaneously. The positive Weyl chamber $\a^+$ is given by the condition 
\begin{align*}
\lambda_1 > \lambda_2 >\ldots > \lambda_{n-1} > | \lambda_{n}|.
\end{align*}

\begin{lemma}[The ``max principle'' for $W(D_n)$]\label{d1}
Suppose that  $\lambda$,  $Y\in\overline{\a^+}$ and that
$\lambda\not=a\,(1,\dots,1,-1)$. 	Suppose that the block of $\lambda_1$ in $\lambda$ has length $1\leq j_0<n$. Suppose also that $\min\{k: (w\cdot Y)_k=y_1\}> j_0$ or that $\{k: (w\cdot Y)_k=y_1\}=\emptyset$. 	Then $\langle \lambda, w\cdot Y \rangle\ < \ \langle \lambda, Y \rangle$.	
\end{lemma}

\begin{proof}Suppose $\lambda$ and $Y$ are as in the statement of the lemma. If $y_1$ appears in $w\cdot Y$ then $\langle \lambda, w\cdot Y \rangle\ < \ \langle \lambda, Y \rangle$ by Lemma \ref{principle} so we can assume that only $-y_1$ appears. 

Using the standard property of the Weyl group over $A_n$, $\langle \lambda, w\cdot Y \rangle\ \leq \ \langle \lambda, w_0\,w\cdot Y \rangle$ where $w_0\in S_n$ re-orders the entries of $w\cdot Y$ in decreasing order. The last entry of $w_0\,w\cdot Y $ has to be $-y_1$. 

We first assume $n=2$, or $n\geq 3$ and $j_0\leq n-2$. As $\pm y_n\ge -y_i$ for all $i<n$, we can suppose 
that the $(n-1)$-entry is $-y_i$ for some $i$. Using the element $w_1$ of the Weyl group which changes signs and permutes the last two entries, we have 
$\langle \lambda, w_1\,w_0\,w\cdot Y \rangle-\langle \lambda, w_0\,w\cdot Y \rangle=(\lambda_{n-1}+\lambda_n)\,(y_1+y_i)\geq0$.
It is easy to check that the last inequality is strict if $n=2$.
Finally, by another application of Lemma \ref{principle}, $\langle \lambda, w_1\,w_0\,w\cdot Y \rangle<\langle \lambda, Y \rangle$ and the result follows.

We next handle the case $j_0\geq n-1$, with $n\geq 3$. Let $\lambda=(a,\dots,a,b)$ with $b\in(-a,a]$. and $n\geq 3$.
We will show that $\Delta=\langle \lambda, Y \rangle -\langle \lambda,w_0\,w\cdot Y \rangle >0$.
If $-y_n$ appears in $w_0 w\cdot Y$, we have, using $\sum_{i\not=1,n} a\, y_i \geq \sum_{i\not=1,n} a\, (\pm y_i)$, 
\begin{align*}
\Delta=\langle \lambda, Y \rangle - \langle \lambda, w_0 w\cdot Y \rangle \geq a\, y_1 +b\, y_n -[a\, (-y_n) + b\, (-y_1) ]
=(a+b)( y_1 + y_n)>0
\end{align*}
where we used the hypothesis $b\not=-a$ and the fact that $ y_1 + y_n>0$ (otherwise $-y_n=y_1$ appears in $w_0 w\cdot Y$).
If $-y_n$ does not appear in $w_0 w\cdot Y$, another $-y_k$  appears among the $n-1$ first entries of $w_0 w\cdot Y$. 
This time, we obtain
$
\Delta \ge(a+b)\,y_1 + a\,(y_k-y_n)+a\,y_k+b\,y_n >0
$,
where we used $y_1>0$ (as $Y\not=0$), the hypothesis $a+b>0$, and the inequalities $y_k\ge y_n$, $a\,y_k\geq|b\, y_n|$.
\end{proof}

\begin{lemma}\label{d2}
\end{lemma}

\begin{proof}
Note that 
$\langle \lambda,Y\rangle=(n-1)\,a\,b+a\,b=n\,a\,b$. The only way that $\langle \lambda, w\cdot Y\rangle=n\,a\,b$ is if 
$w\cdot Y=Y$ \emph{i.e.} $w\in W_Y=W_\lambda$.
\end{proof}

\begin{prop}
Property Killing-max is verified in the case of the root system $D_n$.
\end{prop}

\begin{proof}
We proceed by induction on $n\geq 2$. Given Lemma \ref{d2}, if both $\lambda$ and $Y\in \R\,(1,\dots,1,-1)$ then there is nothing to prove. Given the symmetry of the problem, if $\lambda\in \R\,(1,\dots,1,-1)$ and $Y\not\in \R\,(1,\dots,1,-1)$, we can switch their roles
and suppose that $\lambda\notin \R\,(1,\dots,1,-1)$.

The base case $n=2$, in which, by Lemma \ref{d2}, we can assume that $\lambda=(\lambda_1,\lambda_2)\not\in \R(1,-1)$, is clear by inspection.

Assume the result true for $n-1$, $n\geq3$. As explained above, we may assume that $\lambda\not \in \R(1,\dots,1,-1)$. By Lemma \ref{d1}, the equality \eqref{main}
implies that ``$\max \lambda$ meets $\max Y$''. As in the case $A_n$, it follows that there exist permutations $\sigma\in W_\lambda$ and $\gamma\in W_Y$ such that
$(\sigma w\gamma \cdot Y)_1=y_1$. We consider $\tilde\lambda_1=(\lambda_2,\lambda_3,\dots,\lambda_n)$, $\tilde Y_1=(Y_2,\dots,Y_n)$ and $\tilde{w}_1=\sigma w\gamma |_{\tilde\a}$ where $\tilde\a=\{(x_2,\dots, x_n)| \ X=(x_i)_{i\ge 1}\in \a\}$ and we use the induction hypothesis or Lemma \ref{d2} depending on the situation.
\end{proof}

\subsection{Type $F_4$}
We use  Helgason \cite{Helgason2}
and some simple facts about the Weyl group $W=W(F_4)$ from \cite{Cahn}.
We consider the simple roots $\alpha_1={\bf e}_2-{\bf e}_3$, $\alpha_2={\bf e}_3-{\bf e}_4$, $\alpha_3={\bf e}_4$ and $\alpha_4=({\bf e}_1-{\bf e}_2-{\bf e}_3-{\bf e}_4)/2$ and the corresponding reflections $s_{\alpha_i}=s_i$. It follows that 
$\a^+=\{(x_1,x_2,x_3,x_4)\colon x_1>x_2+x_3+x_4,~x_2>x_3>x_4>0\}.$

Denote $\alpha_{12}={\bf e}_1-{\bf e}_2$ and $s_{12}=s_{\alpha_{12}}$. Note that $\alpha_{12}=\alpha_2+2\,\alpha_3+\alpha_4$ is a positive root. It is easy to check that 
\begin{equation}\label{a12}
s_3\,s_4\,s_{12}=s_2\,s_3\,s_4
\end{equation}
by inspection or using \cite[Table 1]{Cahn} on the basis $({\bf e}_i)$.

Let $X=(x_1,x_2,x_3,x_4)$ with $x_1\geq x_2\ge x_3 \ge x_4 \ge 0$, \emph{i.e.} $X\in \overline{\a^+(B_4)}$. We define 
$	W_X^{B_4}\subset W(B_4)$ as the subgroup generated by a subset of the symmetries $s\in 
\{s_{12}, s_1,s_2,s_3\}$ such that $s(X)=X$.

%%%%%%%%%%%%%%%%%%%%%%%%%%%%%%%%%%%%%%%%%%%%%%%%%%%%%%%%%%%%%%%%
\begin{lemma}\label{weyls}
Let $\lambda\in \overline{\a^+(B_4)}$. Then	$W_\lambda^{B_4} \subset W_\lambda$.
\end{lemma}

\begin{proof}
Clear from the definition of $	W_\lambda^{B_4}$.
\end{proof}

Let $\alpha$, $\beta$, $\gamma$ denote the three sets of  roots of $F_4$ defined in \cite[p.{} 85]{Cahn}, with $\alpha=({\bf \pm e}_i)_{i=1}^4$.
Let $ \delta$, $\eta\in \{\alpha, \beta,\gamma\}$ and $W_{ \delta \eta}= \{w\in W\colon w(\delta)=\eta\} $. 
By \cite{Cahn}, we have
$
 W= W_{ \alpha \alpha} \cup W_{ \alpha \beta} \cup W_{ \alpha \gamma}
$.
In order to describe the action of $w\in W$, we define $w_0^\alpha= id, w_0^\beta=s_3 \,s_4$ and $w_0^\gamma=s_4$.
Then, by \cite[Table 1]{Cahn}, we have $w_0^\delta(\alpha)=\delta$ with $\delta\in \{\alpha, \beta,\gamma\}$.
	
The following result is proven in \cite{Cahn}.
Recall that $W(B_4)$ is the group of signed permutations of 4 elements.
%%%%%%%%%%%%%%%%%%%%%%%%%%%%%%%%%%%%%%%
\begin{lemma}\label{WF4}
Let $ \delta\in \{\alpha, \beta,\gamma\}$ and $w\in W_{ \alpha \delta}$. 
There exists $\sigma\in W(B_4)$ such that if
$Y= \sum_{i=1}^4\,y_{i}\,{\bf e}_i$, then
$
w\cdot Y=\sum_{i=1}^4\,y_{\sigma(i)}\,w_0^\delta({\bf e}_i).
$
Equivalently, $(w_0^\delta)^{-1}\,w$ is a signed permutation with respect to the basis $({\bf e}_i)$.
\end{lemma}

%%%%%%%%%%%%%%%%%%%%%%%%%%%%%%%%%%%%%%%%%%%%%%%%%%%

\begin{prop}\label{F4}
Property Killing-max is verified in the case of the root system $F_4$.
\end{prop}
%%%%%%%%%%%%%%%%%%%%%%%%%%%%%%%%%%%%%%%%%
\begin{proof}
Suppose that $\lambda=\sum_{i=1}^4\,\lambda_i\,{\bf e}_i$, $Y=\sum_{i=1}^4\,y_i\,{\bf e}_i\in\overline{\a^+(F_4)}$ are singular. Our objective is to solve the equation \eqref{main}. We will assume from now on that \eqref{main} holds.
We consider the three cases $w\in W_{ \alpha \delta}$, where $ \delta=\alpha,\beta,\gamma$.

If $w\in W_{ \alpha \alpha }$, we note that
$
{\a^+(F_4)}\subset {\a^+(B_4)}.
$
Lemma \ref{WF4}, Proposition \ref{KillBC} and Lemma \ref{weyls} imply that $w\in W_\lambda^{B_4}\,W_Y^{B_4}\subset W_\lambda\,W_Y$.
 
In the case $w\in W_{\alpha \beta}$, we use $w_0=w_0^\beta= s_3 \,s_4$.
If $\lambda=\sum_{i=1}^4\,\lambda_i\,{\bf e}_i\in\overline{\a^+(F_4)}$ then $\lambda'=w_0^{-1}\cdot\lambda
=\sum_{i=1}^4\lambda'_i \,{\bf e}_i$ with 
$\lambda'_1\ge \lambda'_2\ge \lambda'_3 \ge \lambda'_4 \ge 0$ since
\begin{equation}\label{2b}
\lambda'= 
\frac{1}{2}[ (\lambda_1+\lambda_2+\lambda_3-\lambda_4)\,{\bf e}_1+ 
 (\lambda_1+\lambda_2-\lambda_3+\lambda_4)\,{\bf e_2}
+
(\lambda_1-\lambda_2+\lambda_3+\lambda_4)\,{\bf e}_3+
(\lambda_1-\lambda_2-\lambda_3-\lambda_4)\,{\bf e}_4].
\end{equation}
Using \eqref{main}, Lemma \ref{WF4} and the standard property of the Killing form for $B_4$, we have
\begin{align*}
\langle \lambda, Y\rangle=\langle \lambda, w\cdot Y\rangle=\langle w_0^{-1}\cdot\lambda, w_0^{-1}\,w\cdot Y\rangle
\leq \langle w_0^{-1}\cdot\lambda, Y\rangle = \langle \lambda, w_0\cdot Y\rangle \leq \langle\lambda,Y\rangle.
\end{align*}
This means that $\langle \lambda', w_0^{-1}\,w\cdot Y\rangle
= \langle\lambda', Y\rangle $ and therefore that $w\in w_0\,W_{\lambda'}^{B_4}\,W_{Y}^{B_4} $ by Proposition \ref{KillBC}.
 
We reason similarly if $w\in W_{ \alpha \gamma }$, with $w_0=w_0^\gamma=s_4$ and 
\begin{align*}
\lambda'&=s_4(\lambda)
=\frac{1}{2}[(\lambda_1+\lambda_2+\lambda_3+\lambda_4)\,{\bf e}_1+ (
\lambda_1+\lambda_2-\lambda_3-\lambda_4)\,{\bf e}_2 +(\lambda_1-\lambda_2+\lambda_3-\lambda_4)\,{\bf e}_3
\\&\qquad
+(\lambda_1-\lambda_2-\lambda_3+\lambda_4)\,{\bf e}_4].
\end{align*}
It therefore follows that $w\in w_0\,W_{\lambda'}^{B_4}\,W_{Y}^{B_4} $ with $\lambda'=w_0^{-1}\cdot\lambda$.

It is important to note that a feature of both cases $w\in W_{ \alpha \beta }$ and $w\in W_{ \alpha \gamma }$ implies that the respective $w_0$ satisfy 
$\langle\lambda,Y\rangle=\langle\lambda,w_0\cdot Y\rangle$.
It follows that these cases do not occur if $\alpha_4\not \in \Sigma_\lambda\cup \Sigma_Y$. Indeed,
using the formula $s_i(X)=X-2\,\frac{\alpha_i(X) }{\|\alpha_i\|^2}\alpha_i$, we have for $w_0^\beta=s_3\,s_4$ and for $w_0^\gamma=s_4$,
\begin{align}
\langle\lambda,Y\rangle- 	\langle\lambda,s_3\,s_4\,Y\rangle&=2\,\alpha_4(\lambda)\,\alpha_4(Y) +2 \,\alpha_3(\lambda)\,\alpha_3(Y) + 2\,\alpha_3(\lambda)\,\alpha_4(Y)
\label{step1}\\
\langle\lambda,Y\rangle- 	\langle\lambda,s_4\,Y\rangle&=2\,\alpha_4(\lambda)\,\alpha_4(Y)
\nonumber
\end{align}
Thus $\langle\lambda,Y\rangle \not= 	\langle\lambda,w_0\cdot Y\rangle$ if
$\alpha_4\not \in \Sigma_\lambda\cup \Sigma_Y$ and $w\in W_{\alpha\beta}$ or $w\in W_{\alpha\gamma}$.
We showed above that in  the case $w\in W_{\alpha\alpha}$, formula \eqref{main} implies that 	$w\in W_\lambda\,W_Y$. The Proposition is thus proven for $\alpha_4\not \in \Sigma_\lambda\cup \Sigma_Y$.

It remains to treat	the cases $\alpha_4 \in \Sigma_\lambda$ or $\alpha_4 \in \Sigma_Y$. 
By symmetry of the problem \eqref{main}, it is sufficient to treat the case $\alpha_4 \in \Sigma_\lambda$, for any singular $Y$. 
We assume henceforth that $\alpha_4 \in \Sigma_\lambda$.

We showed above that in  the case $w\in W_{\alpha\alpha}$, formula \eqref{main} implies that 	$w\in W_\lambda\,W_Y$.

If $w\in W_{\alpha\gamma}$, we have $w_0=w_0^\gamma=s_4$ and therefore $\lambda'=s_4\cdot \lambda=\lambda$ since $\alpha_4 \in \Sigma_\lambda$. Since $s_4\in W_\lambda$, we have
$ w\in s_4 W_{\lambda'}^{B_4}W_Y^{B_4}= s_4 W_\lambda^{B_4}\,W_Y^{B_4} \subset W_\lambda\,W_Y$.

Suppose that $w\in W_{\alpha\beta}$ and recall that $w_0=w_0^\beta=s_3\,s_4$. By \eqref{step1}, we have the following two cases:

(A) $\alpha_3(\lambda)=0$ \ \ \ or \ \ \  (B) $\alpha_3(\lambda)\not=0$, $\alpha_3(Y)=0$ and $\alpha_4(Y)=0$.

In the case (A), we have $w_0^{-1}\cdot \lambda=\lambda$ \emph{i.e.} $\lambda'=\lambda$ and $s_3\,s_4\in W_\lambda$. Therefore, we have
$w\in s_3\,s_4 W_{\lambda'}^{B_4}\,W_Y^{B_4}= s_3\,s_4 \,W_\lambda^{B_4}\,W_Y^{B_4} \subset W_\lambda\,W_Y$.

In the case (B), we compute using \eqref{2b}, $\lambda'=(\lambda_2+\lambda_3, \lambda_2+\lambda_4,\lambda_3+\lambda_4,0)$,
where $\lambda_4>0$. 
We will be using $s_3$ defined by $s_3(x_1,x_2,x_3,x_4)=(x_1,x_2,x_3,-x_4)$. Note that $s_3\cdot Y=Y$ since $y_4=\alpha_3(Y)=0$,
and that $s_3$ commutes with $s_1$ and $s_{12}$. We consider the following mutually exclusive cases (B1)--(B4):

(B1) $\Sigma_\lambda=\{ \alpha_4 \}$: in that case, $W_{\lambda'}^{B_4}=\{id,s_3\}$ and $w\in s_3\,s_4 W_{\lambda'}^{B_4}W_Y^{B_4}\subset W_Y$.

(B2) $\Sigma_\lambda=\{ \alpha_1, \alpha_4 \}$, \emph{i.e.} $\lambda_2=\lambda_3>\lambda_4>0$: in that case, 
$W_{\lambda'}^{B_4}=\{id, s_1\} \, \{id, s_3\}$.

Since $s_1$ commutes with $s_3$ and $s_4$, we have 
$w\in \{id, s_1 \}\,s_3\,s_4\,\{id, s_3\} \,W_Y^{B_4}	 \subset W_\lambda\,W_Y$.

(B3) $\Sigma_\lambda=\{ \alpha_2, \alpha_4 \}$, \emph{i.e.} $\lambda_3=\lambda_4>0$: in that case, 
$W_{\lambda'}^{B_4}=\{id, s_{12} \}\, \{id, s_3\}$. 
Using \eqref{a12}, we find that $w\in \{id, s_{2} \} s_3\,s_4 \{id, s_3\}\,W_Y^{B_4}	 \subset W_\lambda\,W_Y$ .

(B4) $\Sigma_\lambda=\{ \alpha_1, \alpha_2, \alpha_4 \}$, \emph{i.e.} $\lambda_2=\lambda_3=\lambda_4>0$: in that case, 
$W_{\lambda'}^{B_4}=\{id, s_{12}, s_1, s_{12}s_1,s_1 s_{12}, s_1 s_{12} s_1 \}\, \{id, s_3\}$. 

Similarly as in (B2) and (B3), we verify that $s_3\,s_4 W_{\lambda'}^{B_4} \subset W_\lambda\,W_Y$. For example,
$s_3\,s_4(s_1 s_{12} s_1)=s_1\,s_3\,s_4 s_{12} s_1= s_1 \,s_2\,s_3\,s_4 \,s_1= s_1 \,s_2 \,s_1\,s_3\,s_4 \in W_\lambda\,W_Y$.
Thus \eqref{main} implies that $w\in s_3\,s_4\, W_{\lambda'}^{B_4} \,W_{Y}^{B_4}\subset W_\lambda\,W_Y$. 	
\end{proof}

\subsection{Type $G_2$} 

The Cartan space is given by $\a(G_2)=\{ H_{A,B}=(A,B,A-B,0,B-A,-B,-A)\,|\ A,B\in\R\}$ and two simple positive roots are $\alpha(H_{A,B})= A-B$ and 
$\beta(H_{A,B})= B-(A-B)=2B-A$. Consequently, the positive Weyl chamber is given by $\a^+=\{ H_{A,B}\,|\ A>B>A-B>0\}$.

Note that it is sufficient to work on the space $\a=\{h_{A,B}=(A,B,A-B)\colon A,B\in\R\}$ which is isomorphic to $\a(G_2)$. We will work on this space $\a$ from now on. Observe also that the Weyl group $W$ is generated by $s_\alpha$ which interchanges the first two entries and changes the sign of the third and $s_\beta=(2,3)$, so it is included in $S_3\rtimes \{1,-1\}^3$. This inclusion is strict: the group $W$ has 12 elements and $S_3\rtimes \{1,-1\}^3$ has $6\times 2^3=48$ elements.

\begin{prop}\label{G2}
Property Killing-max is verified in the case of the root system $G_2$.
\end{prop}

\begin{proof}
Given that the root system is of rank 2, we only need to consider three cases of singular $\lambda$ and $Y$:
\begin{itemize}
\item[(C1)] $\alpha(\lambda)=\alpha(Y)=0$:
We have $\lambda=(l,l,0)$, $Y=(y,y,0)$, $l$, $y>0$ and $\langle \lambda, w\cdot Y \rangle = \langle \lambda, Y \rangle=2\,l\,y$. It follows that 0 in $Y$ cannot change position in $w\cdot Y$ and no $y$ can become $-y$, so $w\cdot Y=Y$ and $w\in W_Y$.
\item[(C2)] $\alpha(\lambda)=\beta(Y)=0$: We have $\lambda=(l,l,0)$, $Y=(2\,y,y,y)$, $l$, $y>0$ and $\langle \lambda, w\cdot Y \rangle = \langle \lambda, Y \rangle=3\,l\,y$. Then no minus sign is possible in the first two terms of $w\cdot Y$ and
$2\,y$ cannot go to the third position. Consequently, using the fact that $ (h_{A,B})_3= (h_{A,B})_1- (h_{A,B})_2$, we find that $w\cdot Y= (2\,y,y,y)=Y$ (so  $w\in W_Y$) or $w\cdot Y= (y,2\,y,-y)=s_\alpha Y$, which implies that $ s_\alpha w \in W_Y$ and $w\in  s_\alpha  W_Y \subset W_\lambda W_Y$.
\item[(C3)] $\beta(\lambda)=\beta(Y)=0$: We have $\lambda=(2\,l,l,l)$, $Y=(2\,y,y,y)$, $l$, $y>0$. Then $2\,y$ must remain in the first position
in $w\cdot Y$ and no sign change can happen, thus $w\cdot Y=Y$ and $w\in W_Y$.
\end{itemize}
\end{proof}


\begin{thebibliography}{99}
\bibitem{Anker1} J.-P.{} Anker and L.{} Ji.
\textit{Heat Kernel and Green Function Estimates on Noncompact Symmetric Spaces}, Geometric and Functional Analysis, 1999, v. 9, n. 6, 1035--1091. 

\bibitem{Barlet} D.{} Barlet and J.{} L.{} Clerc. \textit{Le comportement \`a l'infini des fonctions de Bessel g\'en\'eralis\'ees, I}, Adv.{} in Math., 61 (1986), 165--183.
%%%%%%%%%%%%%%%%%%%%%%%%%%%%%%%%%%%%%%%%
\bibitem{Bogdan}
K.{}
Bogdan, J. Dziuba\'nski, K. Szczypkowski,
\textit{ Sharp Gaussian estimates for heat kernels of Schr\"odinger operators,}
Integral Equations Operator Theory 91 (2019), no. 1, Paper No. 3, 20 pp.

%%%%%%%%%%%%%%%%%%%%%%%%%%%%%%%%%%%%%%%%%%%%
\bibitem{Bogdan2}
K. Bogdan, W. Hansen, T. Jakubowski, 
 \textit{ Localization and Schr\"odinger perturbations of kernels}, Potential Anal. 39 (2013), no. 1, 13-28
%%%%%%%%%%%%%%%%%%%%%%%%%%%%%%%%%%%%%%%%%%%%
\bibitem{BC14} A.{} Borodin, I.{} Corwin,  \textit{Macdonald processes}, Probab.{} Theory Relat.{} Fields 158, 2014, 225--400.

\bibitem{Cahn} P.{} Cahn, R.{} Haas, A.{} G.{} Helminck, J.{} Li and J.{} Schwartz. \textit{Permutation notations for the exceptional Weyl group $F_4$}, INVOLVE 5:1 (2012), 81--89.

\bibitem{Chung} K.{} L.{} Chung and Z.{} Zhao. \textit{From Brownian Motion to Schr\"odinger's equation}, Grundlehren der mathematischen Wissenschaften, Springer, Volume 312, 1995.
%%%%%%%%%%%%%%%%%%%%%%%%%%%%%%%%%
\bibitem{Denk}
Z. Denkowska, M. Denkowski, J. Stasica,  {\it Ensembles sous-analytiques \`a la polonaise: avec une introduction aux fonctions et ensembles analytiques.} Paris : Editions Hermann, 2008. 144 p.
%%%%%%%%%%%%%%%%%%%%%%%%%%%%%%%%%%




%%%%%%%%%%%%%%%%%%%%%%%%%
\bibitem{dJ} M.{} de Jeu. \textit{Paley-Wiener theorems for the Dunkl transform}, Trans.{} Amer.{} Math.{} Soc.{} 358, 2006, 4225--4250.

%%%%%%%%%%%%%%%%%%%%%%%%%%%%%%%%%%%%%%
\bibitem{Du0}
C.{} F.{} Dunkl,
\textit{Intertwining Operators Associated to the Group $S_3$},
Transactions of the American Mathematical Society, Vol.{} 347, No.{} 9 (Sep., 1995), 3347--3374.

\bibitem{Du1}
C.{} F.{} Dunkl, \textit{Integral kernels with reflection group invariance}, Canad.{} J.{}
Math.{} 43 (1991), 1213--1227.

\bibitem{DuX} C.{} F. Dunkl and Y. Xu, \textit{Orthogonal polynomials of several variables},
Encyclopedia of Mathematics and its Applications, 81.{} Cambridge
University Press, Cambridge, 2001.


%%%%%%%%%%%%%%%%%%%%%%%%%%%%%%%%%%%
\bibitem{Dynkin} E. B. Dynkin,
Markov Processes. vols. 1 and 2.  Academic Press, New York; Springer, Berlin, 1965.
%%%%%%%%%%%%%%%%%%%%%%%%%%%%%%%%%

\bibitem{Dyson} F. Dyson: \rm A Brownian Motion Model for the Eigenvalues of a Random Matrix. {\it J. Math. Phys.} 3, 1191--198, 1962.


\bibitem{Erd} K.{} Erdmann and M.{} J.{} Wildon. \textit{Introduction to Lie Algebras}, Springer 2006

\bibitem{Gallardo2} L.{} Gallardo and C.{} Rejeb. \textit{Newtonian potentials and subharmonic functions associated to root systems}, Potential Anal. 47(2017), 369--400.

\bibitem{Gangolli} R{.} Gangolli. \textit{Asymptotic behaviour of spectra of compact quotients
of certain symmetric spaces}, Acta Math{.} 121, 1968, 151--192

\bibitem{Grabiner} D.{} J.{} Grabiner. \textit{Brownian motion in a Weyl chamber, non-colliding particles, and random matrices}, 
Annales de l'I.H.P. Probabilit\'es et statistiques, 1999 vol.{} 35, no.{} 2, 177--204 .

\bibitem{Graczyk} P.{} Graczyk, T.{} Luks and M.{} R\"osler. 
\textit{On the Green function and Poisson integrals of the Dunkl Laplacian}, Potential Anal 48(2018), 337--360. 

\bibitem{PGTLPS} P.{} Graczyk, T.{} Luks and P. Sawyer. \textit{Potential kernels for radial Dunkl Laplacians}, arXiv:1910.03105,  1--31, 2019.

\bibitem{PGPS1} P. Graczyk and P. Sawyer. \textit{The Convolution of orbital measures on symmetric spaces: a survey},
Proceedings of the Conference Probability on Algebraic and Geometric Structures, Contemporary Mathematics, Vol. 668, 81-110, 2016.

\bibitem{Harish-Chandra} Harish-Chandra. \textit{Differential operators on a semisimple Lie algebra}, Amer.{}
J.{} Math.{} 79 (1957), 241--310.

\bibitem{Helgason1} S. Helgason, \textit{The bounded spherical functions on the Cartan motion group}, 	arXiv:1503.07598, 1--7, 2015.

\bibitem{Helgason2} S. Helgason, \textit{Differential Geometry, Lie Groups and Symmetric spaces}, 
Graduate Studies in Mathematics, 34, American Mathematical Society, Providence, RI, 2001.

\bibitem{Helgason3} S. Helgason, \textit{Groups and geometric analysis. Integral geometry, invariant differential operators, and spherical functions}, Mathematical Surveys and Monographs, 83, American Mathematical Society, Providence, RI, 2000.
%%%%%%%%%%%%%%%%%%%%%%%%%%%%%%%%%%%%
\bibitem{Kall}
O. Kallenberg, Foundations of Modern Probability		Springer Series in Statistics, Probability and Its Applications. 		Springer 	1997.
%%%%%%%%%%%%%%%%%%%%%%%%%%%%%%%%%%
\bibitem{Kamel}
J.{} El Kamel and Ch.{} Yacoub. \textit{Poisson integrals and Kelvin transform associated to Dunkl-Laplacian
operator}, GJPAM (Global Journal of Pure and Applied Mathematics) Vol. 3, nr. 3, 251-261 2007. 

\bibitem{Karlin} K.{} S. Karlin and J.{} McGregor. \textit{Coincidence probabilities}, Pacific J.{} Math.{} 9 (1959), 1141--1164.


\bibitem{Katori}
M.{} Katori.
\textit{Bessel Processes, Schramm--Loewner Evolution, and the Dyson Model}, Springer Briefs in Mathematical Physics, Vol.11, 2016.

\bibitem{K12} M.{} Katori. \textit{Survival probability of mutually killing Brownian motion and the O'Connell process}, J.{} Stat.{} Phys.147, 2012, 206--223.

\bibitem{KT}
M.{} Katori, H.{} Tanemura, \textit{Symmetry of matrix-valued stochastic processes and noncolliding diffusion particle systems}, J. Math. Phys.{} 45, 2004, 3058--3085.

\bibitem{Koenig} W.{} K\"onig and N.{} O'Connell, \textit{Eigenvalues of the Laguerre process as non-colliding squared Bessel process}, Elec.{} Comm.{}  Probab.{}  6, 2001, 107--114.

\bibitem{Krantz} S.{} G.{} Krantz and H.{} R.{} Parks. \textit{A Primer of Real Analytic Functions}, Birkh\"auser Adavnced Texts, Basler Lehrb\'ucher, Second Edition, 2002.

\bibitem{Macdonald} I.{} G.{} Macdonald. \textit{Some conjecture for root systems}, SIAM J.{} MATH.{} ANAL.{} Vol.{} 13, No.{} 6, 1982.

\bibitem{Moy} A.{} Moy. \textit{Minimal $K$-types for $G_2$ over a $p$-adic field\/}, Trans.\ Amer.\ Math.\ Soc.\ 305 (1988), no. 2, 517--529.

\bibitem{Narayanan} E.\ K.{}  Narayanan, A. Pasquale and S. Pusti. \textit{Asymptotics of Harish-Chandra expansions, bounded hypergeometric functions associated with root systems, and applications}, Advances in Mathematics, 252 (2014), 227--259.

\bibitem{O12} N.{} O'Connell. \textit{Directed polymers and the quantum Toda lattice}, Ann.{} Probab.{} 40, 2012,  437--458.

\bibitem{Opdam} E.{} M.{} Opdam. \textit{Some applications of hypergeometric shift operators}, Invent. math. 98, 1989, 1--18.

\bibitem{ry} D.{} Revuz and M.{} Yor. \textit{Continuous martingales and Browmian Motion}  Third Edition, Springer, 2005.

\bibitem{R0} M.{} R\"osler.  \textit{Generalized Hermite Polynomials and the Heat Equation for Dunkl Operators}, Commun. Math. Phys. 192, 519--541, 1998.

\bibitem{RV} M.{} R\"osler and M.{} Voit. \textit{Dunkl theory, convolutions algebras, and related Markov processes}, in: Harmonic and stochastic analysis of Dunkl processes (eds.{} P.{} Graczyk et al.).  Travaux en cours 71, 2008.

\bibitem{Roesler} M.{} R\"osler and M.{} Voit. \textit{Positivity of Dunkl's intertwining operator via the trigonometric 
setting}, Int.{} Math.{} Res.{} Not. 2004, no.{} 63, 3379--3389.

\bibitem{Rose} H.{} E.{} Rose. \textit{A Course on Finite Groups}, Universitext, Springer, 2010.

\bibitem{Schapira} B.{} Schapira. \textit{Contributions to the hypergeometric function theory of Heckman and Opdam: sharp estimates, Schwartz space, 
heat kernel}, Geom.\ Funct.\ Anal.\ 18 (1) (2008) 222--250.

\bibitem{Wolf} 
J.{} A.{} Wolf.  \textit{Spherical functions on Euclidean space}, J. Funct. Anal. 239
(2006) 127-136.

\bibitem{Xu} J.{} Xu.  \textit{The bounded spherical functions on the Cartan Motion group and Generalizations for the eigenspaces of the Laplacian on $\R^n$}
 arXiv:1608.05500v1 (2016)
\end{thebibliography}
\end{document}